\documentclass{conm-p-l}
\usepackage{amssymb}
\usepackage{amsmath}
\usepackage{amsthm}
\usepackage[all,knot,cmtip]{xy}
\xyoption{arc}
\entrymodifiers={+!!<0pt,\fontdimen22\textfont2>} 
\usepackage{units}
\usepackage{enumitem}
\setenumerate[1]{leftmargin=2.5em}
\usepackage{mathrsfs}

\numberwithin{equation}{section}

\theoremstyle{plain}   
\newtheorem{bigthm}{Theorem}   

\newtheorem{theorem}[equation]{Theorem}  
\newtheorem{corollary}[equation]{Corollary}     
\newtheorem{lemma}[equation]{Lemma}         
\newtheorem{proposition}[equation]{Proposition} 

\theoremstyle{definition}
\newtheorem{definition}[equation]{Definition}

\theoremstyle{remark}
\newtheorem{remark}[equation]{Remark}

\newcommand{\Fil}{\operatorname{Fil}}
\newcommand{\Aut}{\operatorname{Aut}}
\newcommand{\Hom}{\operatorname{Hom}}
\newcommand{\Gal}{\operatorname{Gal}}
\newcommand{\Spec}{\operatorname{Spec}}
\newcommand{\TF}{\operatorname{TF}}

\newcommand{\TR}{\operatorname{TR}}
\newcommand{\THH}{\operatorname{THH}}
\newcommand{\TP}{\operatorname{TP}}
\newcommand{\HH}{\operatorname{HH}}
\newcommand{\HP}{\operatorname{HP}}

\newcommand{\Z}{\mathbb{Z}}
\newcommand{\Q}{\mathbb{Q}}
\newcommand{\C}{\mathbb{C}}
\newcommand{\Zp}{\mathbb{Z}_p}
\newcommand{\Qp}{\mathbb{Q}_p}
\newcommand{\Cp}{C}
\newcommand{\Fp}{\mathbb{F}_p}
\newcommand{\T}{\mathbb{T}}

\newcommand{\holim}{\operatornamewithlimits{holim}}
\newcommand{\cy}{\operatorname{cy}}
\newcommand{\id}{\operatorname{id}}
\newcommand{\tr}{\operatorname{tr}}
\newcommand{\res}{\operatorname{res}}
\newcommand{\Fr}{\operatorname{Fr}}

\begin{document}

\title{Topological Hochschild homology and \\the Hasse-Weil zeta
  function} 

\author{Lars Hesselholt}

\address{Nagoya University, Japan, and University of
Copenhagen, Denmark} 
\email{larsh@math.nagoya-u.ac.jp}

\thanks{Assistance from DNRF Niels Bohr Professorship is gratefully
  acknowledged} 


\maketitle

\section*{Introduction}

In this paper, we consider the Tate cohomology of the circle group
acting on the topological Hochschild homology of schemes. We show that
in the case of a scheme smooth and proper over a finite field, this cohomology
theory naturally gives rise to the cohomological interpretation of the
Hasse-Weil zeta function by regularized determinants envisioned by
Deninger~\cite{deninger}. In this case, the periodicity of the zeta
function is reflected by the periodicity of said cohomology theory,
whereas neither is periodic in general.

Let $X$ be a scheme. Connes' periodic cyclic homology of $X$ may be
defined to be the anticommutative graded ring $\HP_*(X)$ given by the
Tate cohomology groups
$$\HP_i(X) = \hat{H}^{-i}(\T,\HH(X))$$
in the sense of Greenlees~\cite{greenlees,nikolausscholze} of the circle
group $\T$ acting on the Hochschild spectrum
$\HH(X)$. By analogy, we consider the anticommutative graded ring
$\TP_*(X)$ given by the Tate cohomology groups
$$\TP_i(X) = \hat{H}^{-i}(\T,\THH(X))$$
of $\mathbb{T}$ acting on B\"{o}kstedt's topological Hochschild
spectrum $\THH(X)$~\cite{bokstedt,h3}. The latter is
an implementation of Waldhausen's philosophy to replace the initial ring
of algebra $\Z$ by the initial ring of higher algebra
$\mathbb{S}$. The graded ring $\TP_*(X)$ is the abutment of the
Tate spectral sequence
$$E_{i,j}^2 = H^{-i}(\mathbb{P}_{-\infty}^{\,\infty}(\C),\THH_j(X))
\Rightarrow \TP_{i+j}(X),$$
which is multiplicative and converges conditionally in the sense of
Boardman~\cite{boardman}. The $E^2$-term is periodic with periodicity
operator given by multiplication by
$$t = c_1(\mathcal{O}(1)) \in H^2(\mathbb{P}^{\,\infty}(\C),\Z) =
H^2(\mathbb{P}_{-\infty}^{\,\infty}(\C),\Z).$$
In Connes' theory, the periodicity element $t$ is an infinite cycle
and is represented by a unique homotopy class $v \in \HP_{-2}(\Z)$,
the multiplication by which is Connes' $S$-operator. But in the
topologically refined situation, the element $t$ does not in general
survive the spectral sequence, and, if it does, is typically not
represented by a unique homotopy class. In particular, the graded ring 
$\TP_*(X)$ is generally not periodic. However, by contrast to Connes'
theory, the graded ring $\TP_*(X)$ is, in favorable situations,
equipped with meromorphic Frobenius operators
$$\begin{xy}
(-12,0)*+{ \TP_*(X) }="1";
(12,0)*+{ \TP_*(X). }="2";
{ \ar^-(.55){\varphi_p} "2";"1";};
\end{xy}$$
The operator $\varphi_p$ is obtained from the \emph{cyclotomic}
structure of $\THH(X)$ by a process that bears some resemblance to
analytic continuation. In particular, it does not always exist, and,
when it does, its construction requires some work.

Suppose now that $X$ is defined over a finite field $k$ of order
$q = p^r$. In this situation, the periodicity element $t$ is an infinite
cycle and is represented by a preferred homotopy class
$v \in \TP_{-2}(k)$, the definition of which is given in
Section~\ref{dividedbottsection} below. Moreover, there is a canonical
identification of the ring $\TP_0(k)$ with the ring $W$ of $p$-typical
Witt vectors in $k$. It follows that $\TP_*(X)$ is a $2$-periodic
anticommutative graded $W$-algebra. In this situation, the Frobenius
operator $\varphi_p$ is defined after inverting $p$. Moreover,
the $W$-linear operator $\varphi_p^r$ and the weight filtration of
$\TP_*(X)$ determines the action of the geometric Frobenius
$\Fr_q \colon X \to X$ on $\TP_*(X) \otimes_{W,\iota}\C$. 

\begin{bigthm}\label{main}Let $k$ be a finite field of order
$q = p^r$, let $W$ be its ring of $p$-typical Witt vectors, and let
$\iota \colon W \to \C$ be a choice of embedding. If
$f \colon X \to \Spec(k)$ is a smooth and proper morphism of
schemes, then, as meromorphic functions on $\,\C$,
$$\zeta(X,s) =\frac{\det_{\infty}(s\cdot\id - \,\Theta \mid
  \TP_{\operatorname{od}}(X)
  \otimes_{W,\iota}\C)}{\det_{\infty}(s\cdot\id - \,\Theta \mid
  \TP_{\operatorname{ev}}(X) \otimes_{W,\iota}\C)},$$
where $\Theta$ is a $\,\C$-linear graded derivation such that
$q^{\Theta} = \Fr_q^*\,$ and
$\,\Theta(v) = \frac{2\pi i}{\log q} \cdot v$. 
\end{bigthm}

The graded derivation $\Theta$ is determined by the
specified value on the periodicity class $v$ and by its values on 
$\TP_0(X) \otimes_{W,\iota}\C$ and $\TP_1(X) \otimes_{W,\iota}\C$,
which are both finite dimensional $\C$-vector spaces. To specify the
latter, we follow Deninger~\cite[Section~2.12]{deninger} and
decompose the Frobenius $\Fr_q^* = \Fr_{q,s}^* \circ \Fr_{q,u}^*$ in
its semisimple and unipotent parts and define 
$$\Theta = \log_q(\Fr_q^*) = \log_q(\Fr_{q,s}^*) +
\log_q(\Fr_{q,u}^*).$$
Here we use some choice of principal branch of the logarithm on the
semisimple part and the series $\smash{ \frac{1}{\log q}\log(-) }$ on
the unipotent part. Since, by naturality, $\smash{ \Fr_q^*(v) = v }$,
the operator $\Theta$ may be seen as the most general 
solution to $\smash{ q^{\Theta} = \Fr_q^* }$, which is in accordance with
Deninger's philosophy that $\Theta$ should be the infinitesimal
generator of a Frobenius flow $\Fr_t^*$. We will recall the notions of
regularized determinant and anomalous dimension in the final
section of the paper.

The operator $s\cdot\id -\, \Theta$ has anomalous dimension
zero. Therefore, we may fix any choice of a positive real scaling
factor $\delta$ and state Theorem~\ref{main} as
$$\zeta(X,s) = \frac{\det_{\infty}(\delta(s\cdot\id - \,\Theta) \mid
  \TP_{\operatorname{od}}(X)
  \otimes_{W,\iota}\C)}{\det_{\infty}(\delta(s\cdot\id - \,\Theta) \mid
  \TP_{\operatorname{ev}}(X) \otimes_{W,\iota}\C)}.$$
In particular, we may take $\delta = \frac{1}{2\pi}$ as in the
archimedean case~\cite[Theorem~1.1]{connesconsani}, where the scaling
factor is essential.

We outline the proof of Theorem~\ref{main}. The present
proof \hskip-2pt\footnote{\,The recent proof of the K\"{u}nneth formula for
$\TP$ by Blumberg-Mandell~\cite[Theorem~A]{blumbergmandell1} has
facilitated a direct proof given by
Tabuada~\cite[Theorem~3.27]{tabuada}.} uses the cohomological
interpretation of the Hasse-Weil zeta function by crystalline
cohomology established by Berthelot in his
thesis~\cite[Th\'{e}or\`{e}me~VII.3.2.3]{berthelot},
$$\zeta(X,s) = \frac{ \det(\id - q^{-s}\Fr_q^* \mid
  H_{\operatorname{crys}}^{\operatorname{od}}(X/W) \otimes_{W,\iota}
  \C) }{ \det(\id - q^{-s}\Fr_q^* \mid
  H_{\operatorname{crys}}^{\operatorname{ev}}(X/W) \otimes_{W,\iota} 
  \C) }.$$
The comparison theorem of
Bloch-Deligne-Illusie~\cite[Th\'{e}or\`{e}me~II.1.4]{illusie} shows
that, in this formula, the crystalline cohomology 
$\smash{ H_{\operatorname{crys}}^i(X/W) }$ may be replaced by the
hypercohomology $\smash{
  H^i(X,W\Omega_X^{\boldsymbol{\cdot}}) }$ with coefficients in the
de~Rham-Witt complex. Therefore, we have the two spectral sequences of
hypercohomology, one of which gives rise to a natural spectral
sequence of finite dimensional $\C$-vector spaces
$$E_2^{i,j} = (\lim_{n,F}\, H^i(X,W_n\Omega_X^j)) \otimes_{W,\iota}\C
\Rightarrow H_{\operatorname{crys}}^{i+j}(X/W) \otimes_{W,\iota}\C$$
called the \emph{conjugate} spectral sequence. Here, the identification of the
$E_2$-term also uses the higher Cartier isomorphism of
Illusie-Raynaud~\cite[Proposition~III.1.4]{illusieraynaud}. We prove
in Theorem~\ref{hodge} below that, as a consequence of the definition
of $\THH(X)$ in~\cite{blumbergmandell,gh} and the calculation of the
equivariant homotopy groups of $\THH(\mathcal{O}_X)$
in~\cite[Theorem~B]{h}, the same finite dimensional $\C$-vector spaces
appear in the $E^2$-term of a natural spectral sequence
$$E_{i,j}^2 = \bigoplus_{m \in \Z} (\lim_{n,F}\,
  H^{-i}(X,W_n\Omega_X^{j+2m})) \otimes_{W,\iota} \C \Rightarrow
\TP_{i+j}(X) \otimes_{W,\iota}\C$$
converging to the periodic topological cyclic homology of $X$. We call
this spectral sequence the Hodge spectral sequence. The differentials
in the spectral sequence preserve the direct sum decomposition of the
$E^2$-term, and we call the number $w = j+m$ the
weight.\footnote{\,Bhatt-Morrow-Scholze~\cite{bhattmorrowscholze1}
  have now constructed a filtration of $\TP$ by weight, the
  filtration quotients of which are given by crystalline cohomology.}
By comparing the two spectral sequences and using that determinants
are multiplicative on exact sequences, we conclude from Berthelot's
cohomological interpretation that
$$\zeta(X,s) =\frac{ \det(\id - q^{-s}\Fr_q^* \mid
  \TP_1(X) \otimes_{W,\iota} \C) }{ \det(\id - q^{-s}\Fr_q^* \mid
 \TP_0(X) \otimes_{W,\iota} \C) }.$$
We wish to instead express the zeta function by using the full infinite
dimensional graded $\C$-algebra $\TP_*(X)
\otimes_{W,\iota}\C$. Indeed, we contend that this theory precisely is
the infinite dimensional theory produced by
Deninger~\cite{deninger} through an algebraic modification of
the finite dimensional crystalline cohomology theory employing the
Riemann-Hilbert correspondance on $\mathbb{G}_m$. Accordingly, as a
consequence of op.~cit., Corollary~2.8, we conclude that with $\Theta$
as in Theorem~\ref{main},
$$\zeta(X,s) = \frac{ \det_{\infty}(s \cdot \id - \, \Theta \mid
 \TP_{\operatorname{od}}(X) \otimes_{W,\iota} \C) }{ \det_{\infty}(s
 \cdot \id -\,\Theta \mid \TP_{\operatorname{ev}}(X)
 \otimes_{W,\iota} \C) }.$$
This completes the outline of the proof of Theorem~\ref{main}. We
remark that this result gives a cohomological interpretation of
the zeta function $\zeta(X,s)$ itself as opposed to one of the
function $Z(X,t)$ with $\zeta(X,s) = Z(X,q^{-s})$.

We close with a remark concerning periodicity. If $k$ is a
commutative ring, then the Hochschild homology $\HH_*(k)$ is an
anticommutative graded $k$-algebra and the graded ideal $I \subset
\HH_*(k)$ spanned by the homogeneous elements of positive even degree
has a natural divided power structure~\cite[Expos\'{e}~7]{cartan}. In
particular, it is not possible for the canonical $k$-algebra homomorphism from the symmetric
algebra
$$\xymatrix{
{ S_k(\HH_2(k)) } \ar[r] &
{ \HH_*(k) } \cr
}$$
to be an isomorphism, unless $k$ is a $\Q$-algebra. The
topological Hochschild homology $\THH_*(k)$ also is an
anticommutative graded $k$-algebra, but there is no longer a divided
power structure on graded ideal $I \subset \THH_*(k)$ spanned by
homogeneous elements of positive even degree. Hence, the canonical
$k$-algebra homomorphism
$$\xymatrix{
{ S_k(\THH_2(k)) } \ar[r] &
{ \THH_*(k) } \cr
}$$
can be an isomorphism, and B\"{o}kstedt's
periodicity theorem~\cite{bokstedt1} shows that this is indeed the
case for $k = \mathbb{F}_p$; see also~\cite[Theorem~5.2]{hm}. It is
this basic periodicity theorem in conjunction with the cyclotomic
structure of topological Hochschild homology that makes
Theorem~\ref{main} possible.

\section{The Tate spectrum}\label{tatesection}

In this section, we recall Greenlees' generalization of Tate
cohomology~\cite{greenlees}. We refer the reader
to~\cite[Section~4]{hm} for more details and will follow the
conventions therein. A more modern and comprehensive account, which in
particular treats the multiplicative properties of the construction in
detail, is given in~\cite{nikolausscholze}.

Let $G$ be a compact Lie group, let $E$ be a free left
$G$-CW-complex whose underlying space is contractible, and let $f
\colon E_+ \to S^{\,0}$ be the map that collapses $E$ onto $0 \in
\{0,\infty\} = S^{\,0}$. We consider the following cofibration
sequence
$$\xymatrix{
{ E_+ } \ar[r]^-{f} &
{ S^{\,0} } \ar[r]^-{i} &
{ \tilde{E} } \ar[r]^-{\partial} &
{ \Sigma E_+ } \cr
}$$
in the homotopy category of pointed left $G$-spaces. The sequence is
unique, up to unique isomorphism, and the induced sequence of
suspension $G$-spectra is a cofibration sequence in the $G$-stable
homotopy category. We abuse notation and denote the latter sequence by
the same symbols. Now, for $X$ a $G$-spectrum, we consider the diagram
$$\xymatrix{
{ (E_+ \otimes X)^G } \ar[r] \ar[d] &
{ X^G } \ar[r] \ar[d] &
{ (\tilde{E} \otimes X)^G } \ar[r] \ar[d] &
{ \Sigma (E_+ \otimes X)^G } \ar[d] \cr
{ (E_+ \otimes [E_+,X])^G } \ar[r] &
{ [E_+,X]^G } \ar[r] &
{ (\tilde{E} \otimes [E_+,X])^G } \ar[r] &
{ \Sigma (E_+ \otimes [E_+,X])^G } \cr
}$$
where we write ``$\otimes$'' and ``$[-,-]$'' to indicate the symmetric
monoidal product and internal hom object in the homotopy category of
$G$-spectra, respectively. The left-hand vertical morphism is an
isomorphism. We write the lower sequence as
$$\xymatrix{
{ H_{\boldsymbol{\cdot}}(G,S^{\mathfrak{g}} \otimes X) } \ar[r]^-{N} &
{ H^{\boldsymbol{\cdot}}(G,X) } \ar[r]^-{i} &
{ \hat{H}^{\boldsymbol{\cdot}}(G,X) } \ar[r]^-{\partial} &
{ \Sigma H_{\boldsymbol{\cdot}}(G,S^{\mathfrak{g}} \otimes X) } \cr
}$$
and call it the Tate cofibration sequence. Here $S^{\mathfrak{g}}$ is
the one-point compactification of the adjoint representation of $G$ on
its Lie algebra $\mathfrak{g}$.

We are interested mainly in the case, where $G$ is either the full
circle group $\mathbb{T}$ of complex numbers of modulus $1$ or one of
its finite subgroups $C_r \subset \mathbb{T}$. We record the following
result, referring the reader to~\cite{bousfield} for background on
$p$-completion.

\begin{lemma}\label{tatetower}Let $p$ be a prime number and let $X$ be
a $\mathbb{T}$-spectrum. If the underlying spectrum of $X$ is
$p$-complete and bounded below, then the canonical morphism
$$\xymatrix{
{ \hat{H}^{\boldsymbol{\cdot}}(\mathbb{T},X) } \ar[r] &
{ \holim_n \hat{H}^{\boldsymbol{\cdot}}(C_{p^n},X) } \cr
}$$
is a weak equivalence.
\end{lemma}

\begin{proof}It is proved in~\cite[Lemma~2.1.1]{h8} that the morphism
in question becomes a weak equivalence after
$p$-completion. Hence, it suffices to show that the domain and
target of this morphism are $p$-complete. We prove the statement
for the domain; the proof for the target is analogous. We wish to show
that for every integer $i$, the canonical morphism
$$\xymatrix{
{ \pi_i(\hat{H}^{\boldsymbol{\cdot}}(\mathbb{T},X)) } \ar[r] &
{ \pi_i(\hat{H}^{\boldsymbol{\cdot}}(\mathbb{T},X),\Z_p) } \cr
}$$
is an isomorphism. The assumption that the underlying spectrum of $X$
is bounded below implies that there exists a non-negative integer $s$
depending on $i$ such that the morphism
$$\xymatrix{
{ (\Fil_s(\tilde{E}) \otimes [E_+,X])^{\mathbb{T}} }
\ar[r] &
{ (\tilde{E} \otimes [E_+,X])^{\mathbb{T}} } \cr
}$$
induced by the inclusion of the $s$-skeleton induces an isomorphism of
homotopy groups in degrees less than or equal to $i$. We may further
assume that the $\mathbb{T}$-CW-complex $\Fil_s(\tilde{E})$ is
finite.
Therefore, replacing $X$ by $\Fil_s(\tilde{E}) \otimes X$, the
underlying spectrum of which again is $p$-complete, it suffices to
show that the group cohomology spectrum
$H^{\boldsymbol{\cdot}}(\mathbb{T},X)$ is $p$-complete. In this case,
the completion morphism factors as the composition
$$\xymatrix{
{ H^{\boldsymbol{\cdot}}(\mathbb{T},X) } \ar[r] &
{ H^{\boldsymbol{\cdot}}(\mathbb{T},X_p) } \ar[r] &
{ H^{\boldsymbol{\cdot}}(\mathbb{T},X)_p } \cr
}$$
of the morphism induced by the completion morphism for $X$ and a
canonical isomorphism. But the left-hand morphism is a weak
equivalence by the assumption that the underlying non-equivariant
spectrum of $X$ is $p$-complete.
\end{proof}

The skeleton filtrations of the pointed $G$-CW-complexes $E_+$ and
$\tilde{E}$ give rise to spectral sequences converging conditionally
to the homotopy groups of the three terms in the Tate cofibration
sequence. We call the spectral sequence converging to the homotopy
groups of the Tate spectrum the Tate spectral sequence. In the
situation of Lemma~\ref{tatetower}, we take $E = S(\C^{\infty})$ and
$\tilde{E} = S^{\C^{\infty}}$ with the $G$-CW-structures defined
in~\cite[Section~4.4]{hm4}. For the circle group $G = \mathbb{T}$,
the spectral sequence takes the form 
$$E^2 = S\{t^{\pm1}\} \otimes \pi_*(X) \Rightarrow
\hat{H}^{-*}(\mathbb{T},X)$$
with the generator $\smash{ t = c_1(\mathcal{O}(1)) \in E_{-2,0}^2 }$
specified below. For $G = C_{p^n}$, we have
$$E^2 = S\{t^{\pm1}\} \otimes \pi_*(X)/p^n\pi_*(X) \Rightarrow
\hat{H}^{-*}(C_{p^n},X),$$
if the homotopy groups $\pi_*(X)$ are $p$-torsion free, and
$$E^2 = \Lambda\{u\} \otimes S\{t^{\pm1}\} \otimes \pi_*(X)
\Rightarrow \hat{H}^{-*}(C_{p^n},X),$$
if the homotopy groups $\pi_*(X)$ are annihilated by $p$. We define
the generators $\smash{ t \in E_{-2,0}^2 }$ and
$\smash{ u \in E_{-1,0}^2 }$ to be the classes of the cycles specified
in~\cite[Lemma~4.2.1]{hm4}. We now define the generator $t$ for $G =
\mathbb{T}$ to be the unique class that restricts to the generator $t$
for $G = C_{p^n}$. If $X$ is a ring $\mathbb{T}$-spectrum, then the
spectral sequences are multiplicative with the indicated bi-graded
ring structure on the $E^2$-term, except that if $p = 2$ and $n = 1$,
then $u^2 = t$. The restriction maps in Tate cohomology
$$\xymatrix{
{ \hat{H}^{\boldsymbol{\cdot}}(\mathbb{T},X) } \ar[r] &
{ \hat{H}^{\boldsymbol{\cdot}}(C_{p^n},X) } \ar[r] &
{ \hat{H}^{\boldsymbol{\cdot}}(C_{p^{n-1}},X) } \cr
}$$
induce multiplicative maps of Tate spectral sequences that, on the
$E^2$-terms, are given by the tensor product of the canonical
projections
$$\xymatrix{
{ \pi_*(X) } \ar[r] &
{ \pi_*(X)/p^n\pi_*(X) } \ar[r] &
{ \pi_*(X)/p^{n-1}\pi_*(X) } \cr
}$$
and the graded ring homomorphisms that take $t$ to $t$ and $u$ to $0$.

\section{Cyclotomic spectra}\label{cyclotomicsection}

We briefly recall the notion of a $p$-cyclotomic
spectrum, which was introduced in~\cite[Section~2]{hm} and codifies
the structure available on the topological Hochschild spectrum in
addition to the circle action discovered by Connes. A better and more
flexible version of this notion was introduced by Nikolaus and
Scholze~\cite{nikolausscholze}. 

The group $\mathbb{T}$ has the special property that there
exists a group isomorphism
$$\xymatrix{
{ \mathbb{T} } \ar[r]^-{\rho} &
{ \mathbb{T}/C_p, } \cr
}$$
which we will always choose to be the $p$th root. A
$p$-cyclotomic structure on a $\mathbb{T}$-spectrum $T$ was defined
in~\cite[Definition~2.2]{hm} to be a morphism of $\mathbb{T}$-spectra
$$\xymatrix{
{ \rho^*(\tilde{E} \otimes T)^{C_p} } \ar[r]^-{r} &
{ T } \cr
}$$
with the property that for all $n \geqslant 0$, the induced map of
$C_{p^n}$-fixed point spectra is a weak equivalence. The diagram at
the beginning of Section~\ref{tatesection} takes the form
$$\xymatrix{
{ H_{\boldsymbol{\cdot}}(C_{p^n},T) } \ar[r] \ar@{=}[d] &
{ \TR^{n+1}(T;p) } \ar[r]^-{R} \ar[d]^-(.43){\tilde{\gamma}} &
{ \TR^n(T;p) } \ar[r]^-{\partial} \ar[d]^-{\gamma} &
{ \Sigma H_{\boldsymbol{\cdot}}(C_{p^n},T) } \ar@{=}[d] \cr
{ H_{\boldsymbol{\cdot}}(C_{p^n},T) } \ar[r]^-{N} &
{ H^{\boldsymbol{\cdot}}(C_{p^n},T) } \ar[r]^-{i} &
{ \hat{H}^{\boldsymbol{\cdot}}(C_{p^n},T) } \ar[r]^-{\partial} &
{ H_{\boldsymbol{\cdot}}(C_{p^n},T) } \cr
}$$
where $\TR^n(T;p)$ is the $C_{p^{n-1}}$-fixed point spectrum of $T$,
and where the morphism $\varphi$ is the map of $C_{p^{n-1}}$-fixed
point spectra induced by the composition
$$\xymatrix{
{ T } &
{ \rho^*(\tilde{E} \otimes T)^{C_p} } \ar[r]^-{f^*} \ar[l]_-(.46){r} &
{ \rho^*(\tilde{E} \otimes [E_+,T])^{C_p} } \cr
}$$
of an inverse of the morphism $r$ and the morphism $f^*$ in the
original diagram. We remark that, recently, Nikolaus and Scholze have
made the remarkable observation that it is only the composite
$\rho$-equivariant morphism
$$\xymatrix{
{ T } \ar[r]^-{\gamma} &
{ \hat{H}^{\boldsymbol{\cdot}}(C_p,T) } \cr
}$$
that is of importance and that the factorization of this map inherent
in the original definition of a $p$-cyclotomic structure is
immaterial. As demonstrated in~\cite{nikolausscholze}, this
approach makes it possible to develop the theory of cyclotomic
spectra without the use of equivariant stable homotopy
theory. In the following, we will therefore refer to $\varphi$ as the
cyclotomic structure morphism. We call the morphism
$$\xymatrix{
{ \TR^{n+1}(T;p) } \ar[r]^-{R} &
{ \TR^n(T;p) } \cr
}$$
in the diagram above the restriction morphism; and we call the
morphism
$$\xymatrix{
{ \TR^{n+1}(T;p) } \ar[r]^-{F} &
{ \TR^n(T;p) } \cr
}$$
given by the canonical inclusion the Frobenius morphism. We write
$$\TF(T;p) = \holim_{n,F} \TR^n(T;p)$$
for the homotopy limit with respect to the Frobenius morphisms. Using
the $p$-cyclotomic structure on $T$, we now define the inverse
Frobenius operator
$$\xymatrix{
{ \TF(T;p) } \ar[r]^-{\varphi^{-1}} &
{ \TF(T;p) } \cr
}$$
to be the composition
$$\xymatrix{
{ \holim_{n,F} \TR^n(T;p) } \ar[r] &
{ \holim_{n,F} \TR^{n+1}(T;p) } \ar[r] &
{ \holim_{n,F} \TR^n(T;p) } \cr
}$$
of the restriction along the successor functor and the morphism
induced by the restriction morphisms. The cyclotomic structure
morphism induces a morphism
$$\xymatrix{
{ \TF(T;p) = \holim_{n,F} \TR^n(T;p) } \ar[r]^-{\gamma} &
{ \holim_{n,F} \hat{H}^{\boldsymbol{\cdot}}(C_{p^n},T), } \cr 
}$$
and Lemma~\ref{tatetower} identifies the target with
$\hat{H}^{\boldsymbol{\cdot}}(\mathbb{T},T)$, provided that $T$ is
$p$-complete and bounded below. The inverse Frobenius operator
$\varphi^{-1}$ does typically not extend to an operator defined on
$\hat{H}^{\boldsymbol{\cdot}}(\mathbb{T},T)$, even at the level of
homotopy groups.


%
%
%
%

\section{Topological Hochschild homology}\label{thhsection}

In this section, we recall the structure of topological
Hochschild homology for schemes smooth over a perfect field of
characteristic $p > 0$ from~\cite{h}. We also recall B\"{o}kstedt's
periodicity theorem from~\cite{bokstedt1}.

Topological Hochschild homology is defined in the same generality as
algebraic $K$-theory and assigns to an exact category with weak
equivalences a cyclotomic spectrum. In particular, for a 
scheme $X$, applying this theory to the exact category of perfect
complexes of quasi-coherent $\mathcal{O}_X$-modules with
quasi-isomorphisms as weak equivalences, defines a cyclotomic spectrum
$\THH(X)$. The derived tensor product of perfect complexes
gives rise to a pairing on $\THH(X)$, which makes it a commutative
cyclotomic ring
spectrum.\footnote{\,It follows
  from~\cite[Section~IV.2]{nikolausscholze} that $\THH(X)$ is an
  $\mathbb{E}_{\infty}$-algebra in cyclotomic spectra.}
Moreover, for every morphism of schemes $f \colon X \to Y$, left
derived extension of scalars gives rise to a morphism of cyclotomic ring
spectra $f^* \colon \THH(Y) \to \THH(X)$; and for every proper
morphism of schemes $f \colon X \to Y$, right derived restriction of
scalars gives rise to a morphism of cyclotomic $\THH(Y)$-module
spectra $f_* \colon \THH(X) \to \THH(Y)$.  

The homotopy groups $\THH_*(X)$ form an anticommutative graded
ring and we define $\THH_*(\mathcal{O}_X)$ to be the sheaf of
anticommutative graded rings on the small \'{e}tale site of $X$
associated to the presheaf that to an \'{e}tale morphism $f \colon U
\to X$ assigns $\THH_*(U)$. There is a natural ring isomorphism
$$\xymatrix{
{ \mathcal{O}_X } \ar[r]^-{\eta} &
{ \THH_0(\mathcal{O}_X) } \cr
}$$
given by the inclusion of the zero-skeleton, and we will consider
$\THH_*(\mathcal{O}_X)$ as a sheaf of anticommutative 
graded $\mathcal{O}_X$-algebras via this map. As such, it is
quasi-coherent by~\cite[Proposition~3.2.1]{gh}, and there is a
multiplicative and conditionally convergent descent spectral sequence
$$E_{i,j}^2 = H^{-i}(X,\THH_j(\mathcal{O}_X)) \Rightarrow
\THH_{i+j}(X).$$
We consider this spectral sequence, which we call the Hodge
spectral sequence for topological Hochschild homology, in detail in
Section~\ref{hodgesection} below. Since $\THH_j(\mathcal{O}_X)$ is
zero for $j < 0$, it shows, in particular, that $\THH_i(X)$ is zero for $i < -\dim(X)$.  

More generally, we let $r = p^{n-1}$ be a prime power and let
$C_r \subset \mathbb{T}$ be the subgroup of order $r$. The equivariant
homotopy groups
$$\TR_*^n(X;p) = \pi_*(\THH(X)^{C_r})$$
again form an anticommutative graded ring and we define
$\TR_*^n(\mathcal{O}_X;p)$ to be the sheaf of anticommutative graded
rings on the small \'{e}tale site of $X$ associated to the presheaf
that to an \'{e}tale morphism $f \colon U \to X$ assigns
$\TR_*^n(U;p)$. By~\cite[Theorem~F]{hm}, there is a canonical natural
ring isomorphism
$$\xymatrix{
{ W_n(\mathcal{O}_X) } \ar[r]^-{\eta} &
{ \TR_0^n(\mathcal{O}_X;p), } \cr
}$$
which is compatible with the restriction, Frobenius, and
Verschiebung operators, and we view $\TR_*^n(\mathcal{O}_X;p)$ as a 
sheaf of anticommutative graded $W_n(\mathcal{O}_X)$-algebras via this
map. As such, it is quasi-coherent. Indeed, this follows from the
proof of~\cite[Proposition~6.2.4]{hm3} by
using~\cite[Theorem~B]{borger} and~\cite[Corollary~15.4]{borger1},
which were proved later.

The $W_n(\mathcal{O}_X)$-modules $\TR_j^n(\mathcal{O}_X;p)$ together
with the cup product $\cup$, the restriction $R$, the Frobenius $F$, the
Verschiebung $V$, and Connes' operator $d$ constitute a $p$-typical
Witt complex over $\mathcal{O}_X$;
see~\cite[Definition~4.1]{h6}. Accordingly, this $p$-typical Witt
complex receives a unique map of $p$-typical Witt complexes
$$\xymatrix{
{ W_n\Omega_X^j } \ar[r]^-{\eta} &
{ \TR_j^n(\mathcal{O}_X;p) } \cr
}$$
from the $p$-typical de~Rham-Witt complex of $X$, which, by
definition, is the initial $p$-typical Witt complex over
$\mathcal{O}_X$. This map is an isomorphism for $j \leqslant 1$
by~\cite{h9}. 

We next recall the following fundamental periodicity theorem. The
basic case $n = 1$ is B\"{o}kstedt's theorem proved in his
paper~\cite{bokstedt1} which unfortunately remains unpublished. A
detailed outline of B\"{o}kstedt's proof is given
in~\cite[Section~5.2]{hm}, and the proof of the general case $n
\geqslant 1$ is given in~\cite[Theorem~5.5]{hm}.  

\begin{theorem}\label{periodicitytheorem}Let $k$ be a perfect field of
characteristic $p > 0$ and let $n$ be a positive integer. The
canonical map of graded $W_n(k)$-algebras
$$\xymatrix{
{ S_{W_n(k)}(\TR_2^n(k;p)) } \ar[r] &
{ \TR_*^n(k;p) } \cr
}$$
is an isomorphism and the $W_n(k)$-module $\TR_2^n(k;p)$ is free of
rank $1$.
\end{theorem}

In Section~\ref{dividedbottsection}, we define a preferred
$W_n(k)$-module generator $\alpha_n \in \TR_2^n(k;p)$, 
the divided Bott element. It satisfies $F(\alpha_n) = \alpha_{n-1}$ and
$R(\alpha_n) = p \cdot \alpha_{n-1}$, showing that,
in~\cite[Theorem~5.5]{hm}, we can arrange for the units $\lambda_n$ to
be equal to $1$.

We also recall from~\cite[Theorem~B]{h} the following result, which we
will use in Section~\ref{hodgesection} below to construct the Hodge
spectral sequence mentioned in the Introduction.

\begin{theorem}\label{smoothalgebras}Let $k$ be a perfect field of
characteristic $p > 0$, let $n$ be a positive integer, and let $f
\colon X \to \Spec(k)$ be a smooth morphism. In this situation, the
canonical map of graded $W_n(\mathcal{O}_X)$-algebras
$$\xymatrix{
{ W_n\Omega_X^* \otimes_{f^*W_n(k)}f^*\TR_*^n(k;p) } \ar[r] &
{ \TR_*^n(\mathcal{O}_X;p) } \cr
}$$
is an isomorphism.
\end{theorem}

We note that, by comparison, if $k$ is a separably closed field of any
characteristic and if $f \colon X \to \Spec(k)$ is a smooth morphism,
then the canonical map
$$\xymatrix{
{ K_*^M(\mathcal{O}_X) \otimes_{f^*\Z/p^n\Z} f^*K_*(k,\Z/p^n\Z) } \ar[r] &
{ K_*(\mathcal{O}_X,\Z/p^n\Z) } \cr
}$$
is an isomorphism by Voevodsky. If $\operatorname{char}(k) \neq p$,
then $K_*(k,\Z/p^n\Z)$ is a symmetric algebra generated by the Bott
element $\beta_{\epsilon,n} \in K_2(k,\Z/p^n\Z)$
by Suslin~\cite{suslin,suslin1}; and if  $\operatorname{char}(k) = p$,
then $K_*(k,\Z/p^n\Z) = \Z/p^n\Z$ by
Geisser-Levine~\cite{geisserlevine}.

\section{The divided Bott element}\label{dividedbottsection}

%
%
%
%

Let $k$ be a perfect field of characteristic $p > 0$. We proved
in~\cite[Section~5]{hm} that, as an anticommutative graded ring
$\TP_*(k) = S_{W(k)}\{v^{\pm1}\}$ for some $v \in \TP_{-2}(k)$. In
this section, we make a preferred choice of the generator
$v \in \TP_{-2}(k)$ based on the results in~\cite{h8}. We call this
generator the inverse divided Bott element.

We first define the ring isomorphism
$$\xymatrix{
{ W(k) } \ar[r]^-{\tau} &
{ \TP_0(k) } \cr
}$$
by which we consider $\TP_*(k)$ a graded $W(k)$-algebra. Since $k$ is
an $\Fp$-algebra, the Witt vector Frobenius $F \colon W_n(k) \to
W_{n-1}(k)$ is equal to the common composite of the maps in the
following diagram with the horizontal maps induced by the
Frobenius $\varphi \colon k \to k$,
$$\begin{xy}
(0,7)*+{ W_n(k) }="11";
(30,7)*+{ W_n(k) }="12";
(0,-7)*+{ W_{n-1}(k) }="21";
(30,-7)*+{ W_{n-1}(k). }="22";
{ \ar^-{W_n(\varphi)} "12";"11";};
{ \ar^-{R} "21";"11";};
{ \ar^-{R} "22";"12";};
{ \ar^-{W_{n-1}(\varphi)} "22";"21";};
\end{xy}$$
Therefore, we may define a ring homomorphism
$$\xymatrix{
{ \lim_{n,F} W_n(k) } \ar[r]^-{\varphi^{\infty}} &
{ \lim_{n,R} W_n(k) = W(k) } 
}$$
to be the map of limits induced by the map of limit systems that, at
level $n$, is given by the ring homomorphism
$W_n(\varphi^n) \colon W_n(k) \to W_n(k)$. The map $\varphi^{\infty}$
is an isomorphism, because $k$ is perfect. Moreover, the composition
$$\xymatrix{
{ W_n(k) } \ar[r]^-{\eta} &
{ \TR_0^n(k;p) } \ar[r]^-{\gamma} &
{ \hat{H}^{\,0}(C_{p^n},\THH(k)) } \cr
}$$
of the canonical isomorphism from~\cite[Theorem~F]{hm} and the
isomorphism induced by the cyclotomic structure map defines a ring
isomorphism
$$\xymatrix{
{ \lim_{n,F} W_n(k) } \ar[r]^-{\hat{\eta}} &
{ \lim_{n,F} \hat{H}^0(C_{p^n},\THH(k)), } \cr
}$$
where the structure maps in the right-hand limit system are the
restriction maps in Tate cohomology. Finally, we have an isomorphism
of graded rings
$$\xymatrix{
  { \TP_*(k) } \ar[r]^-{\rho} &
{ \lim_{n,F} \hat{H}^{-*}(C_{p^n},\THH(k)), } \cr
}$$
which is also given by restriction maps in Tate cohomology. We now
define the ring isomorphism $\tau$ to be the composite
$\rho^{-1} \circ \hat{\eta} \circ (\varphi^{\infty})^{-1}$. Extension of scalars
along the unique ring homomorphism $f \colon \Fp \to k$ defines a map
of graded $W(k)$-algebras
$$\xymatrix{
{ W(k) \otimes_{W(\Fp)} \TP_*(\Fp) } \ar[r] &
{ \TP_*(k) } \cr
}$$
and that this map is an isomorphism.

We proceed to define the generator $v \in \TP_{-2}(\Fp)$, and, to this
end, we recall the main results of~\cite{h8}. Let $\Cp$ be the
perfectoid field of $p$-adic complex numbers. It is defined, we
recall, to be the completion of an algebraic closure of the field
$\Q_p$ of $p$-adic numbers and is both complete and algebraically
closed. Its valuation ring $\smash{ \mathcal{O}_{\Cp} }$ is a complete
valuation ring with value group $\Q$, the residue field
$k$ of which is an algebraic closure of $\Fp$.

We first recall the results on the algebraic $K$-theory of
$\mathcal{O}_{\Cp}$ proved by Suslin in~\cite{suslin,suslin1}, based on the
rigidity theorems of Gabber~\cite{gabber} and
Gillet-Thomason~\cite{gilletthomason}. The group
$K_j(\mathcal{O}_{\Cp})$ is $p$-divisible if $j$ is positive, and it
is uniquely $p$-divisible if $j$ is both positive and even. Hence, 
the Bockstein homomorphism
$$\xymatrix{
{ K_2(\mathcal{O}_{\Cp},\Zp) } \ar[r]^-{\partial} &
{ T_p(K_1(\mathcal{O}_{\Cp})) = \Hom_{\Z}(\Qp/\Zp,
  K_1(\mathcal{O}_{\Cp})) } \cr 
}$$
from the $p$-adic $K$-group in degree $2$ to the $p$-primary Tate
module of the $K$-group in degree $1$ is an isomorphism. The
resulting map of graded $K_0(\mathcal{O}_{\Cp},\Zp)$-algebras
$$\xymatrix{
{ S_{K_0(\mathcal{O}_{\Cp},\Zp)}(T_p(K_1(\mathcal{O}_{\Cp}))) } \ar[r] &
{ K_*(\mathcal{O}_{\Cp},\Zp), } \cr
}$$
moreover, is an isomorphism. Here the ring $K_0(\mathcal{O}_{\Cp},\Zp)$ is
canonically isomorphic to $\Zp$ and the $p$-primary Tate module
$\Zp(1) = T_p(K_1(\mathcal{O}_{\Cp}))$ is a free $\Zp$-module of rank
$1$. An element $\epsilon \in \Zp(1)$ determines and is determined by
the sequence $\smash{ (\epsilon^{(v)})_{v \geqslant 0} }$ of compatible
$p$-power roots of unity in $\Cp$ defined by
$\smash{ \epsilon^{(v)} = \epsilon(p^{-v}+\Zp) }$ and it is
a generator if and only if the $p$th root of unity
$\smash{ \epsilon^{(1)} }$ is primitive. In this situation, we follow
Thomason~\cite{thomason} and define the Bott element associated with
$\epsilon$ to be the unique generator
$$\beta_{\,\epsilon} \in K_2(\mathcal{O}_{\Cp},\Zp)$$
such that $\partial(\beta_{\,\epsilon}) = \epsilon$.

We next recall the corresponding results for $\TF$-theory
from~\cite{h8}, which, as it turns out, are completely analogous. The
group $\TF_j(\mathcal{O}_{\Cp};p)$ is $p$-divisible if $j$ is
positive, and it is uniquely $p$-divisible if $j$ is both positive and
even. Accordingly, the Bockstein homomorphism induces an isomorphism
$$\xymatrix{
{ \TF_2(\mathcal{O}_{\Cp};p,\Zp) } \ar[r]^-{\partial} &
{ T_p(\TF_1(\mathcal{O}_{\Cp};p)) =
  \Hom_{\Z}(\Qp/\Zp,\TF_1(\mathcal{O}_{\Cp};p)), } \cr
}$$
and it is proved in op.~cit. that the resulting map of graded
$\TF_0(\mathcal{O}_{\Cp};p,\Zp)$-algebras
$$\xymatrix{
{ S_{\TF_0(\mathcal{O}_{\Cp};p,\Zp)}(T_p(\TF_1(\mathcal{O}_{\Cp};p)))
} \ar[r] &
{ \TF_*(\mathcal{O}_{\Cp};p,\Zp) } \cr
}$$
is an isomorphism. Moreover, the ring $\TF_0(\mathcal{O}_{\Cp};p,\Zp)$ is
canonically isomorphic to Fontaine's ring of $p$-adic periods
$$\textstyle{ A_{\operatorname{inf}} =
  \lim_{n,F}W_n(\mathcal{O}_{\Cp}), }$$
which plays a prominent role in $p$-adic Hodge theory. The $p$-primary
Tate module
$$\textstyle{ A_{\operatorname{inf}}\{1\} =
T_p(\TF_1(\mathcal{O}_{\Cp};p)) = \lim_{n,F}
T_p(W_n\Omega_{\mathcal{O}_{\Cp}}^1) }$$
is a free $A_{\operatorname{inf}}$-module of rank 1. However, the
image by the cyclotomic trace map
$$\xymatrix{
{ K_*(\mathcal{O}_{\Cp},\Zp) } \ar[r]^-{\tr} &
{ \TF_*(\mathcal{O}_{\Cp};p,\Zp) } \cr
}$$
of the Bott element $\beta_{\,\epsilon}$ is not a generator of this
$A_{\operatorname{inf}}$-module. Instead, we proved
in~\cite[Theorem~B]{h8} that the image of the Bott element, which, by
abuse of notation, we also denote by $\beta_{\,\epsilon}$, is
divisible by the non-unit
$$\textstyle{
\mu_{\,\epsilon} = [\varepsilon] - 1
= ([\epsilon^{(n)}]_n - 1)_{n \geqslant 1} \in 
\lim_{n,F} W_n(\mathcal{O}_{\Cp}) = A_{\operatorname{inf}} }$$ 
and that the element $\alpha_{\,\epsilon} = \mu_{\,\epsilon}^{-1}
\cdot \beta_{\,\epsilon}$ is a generator of
$\TF_2(\mathcal{O}_{\Cp};p,\Zp)$.

\begin{definition}\label{dividedbottelement}If $\epsilon$ is a
generator of the free $\Zp$-module $\Zp(1)$, then the associated
divided Bott element is the generator
$$\alpha_{\,\epsilon} = \mu_{\,\epsilon}^{-1} \cdot
\beta_{\,\epsilon}$$
of the free $A_{\operatorname{inf}}$-module
$\TF_2(\mathcal{O}_{\Cp};p,\Zp) = A_{\operatorname{inf}}\{1\}$. 
\end{definition}

By functoriality, the left action of the Galois group
$G = \Gal(\Cp/\Qp)$ on $\mathcal{O}_{\Cp}$ induces an action on the
graded ring $\TF_*(\mathcal{O}_{\Cp};p,\Zp)$, and the action on the Bott
element is given by
$$\sigma \cdot \beta_{\,\epsilon} = \beta_{\sigma(\epsilon)} =
\chi(\sigma) \cdot \beta_{\,\epsilon},$$
where $\chi \colon G \to \Aut(\mu_{p^{\infty}}) = \Zp^*$ is the
cyclotomic character. Since $A_{\operatorname{inf}}$ is an integral
domain, we therefore conclude that
$$\sigma \cdot \alpha_{\,\epsilon} = \alpha_{\sigma(\epsilon)} =
\chi(\sigma) \cdot \sigma(\mu_{\,\epsilon})^{-1} \cdot
\mu_{\,\epsilon} \cdot \alpha_{\,\epsilon}.$$
Similarly, the inverse Frobenius operator induces a homorphism of
graded rings
$$\xymatrix{
{ \TF_*(\mathcal{O}_{\Cp};p,\Zp) } \ar[r]^-{\varphi^{-1}} &
{ \TF_*(\mathcal{O}_{\Cp};p,\Zp). } \cr
}$$
In degree zero, it is given, up to canonical isomorphism, by the ring
automorphism
$$\xymatrix{
{ A_{\operatorname{inf}} } \ar[r]^-{\varphi^{-1}} &
{ A_{\operatorname{inf}} } \cr
}$$
defined as the composition
$$\begin{xy}
(0,0)*+{ \lim_{n,F} W_n(\mathcal{O}_{\Cp}) }="1";
(35,0)*+{ \lim_{n,F} W_{n+1}(\mathcal{O}_{\Cp}) }="2";
(75,0)*+{ \lim_{n,F} W_n(\mathcal{O}_{\Cp}) }="3";
{ \ar^-{\res_{S}} "2";"1";};
{ \ar^-{\lim_{n,F} R} "3";"2";};
\end{xy}$$
of the restriction along the successor functor and the map induced by
the Witt vector restriction. The cyclotomic trace map takes values in
the sub-graded ring of elements fixed by the inverse Frobenius
operator. In particular,
$$\varphi^{-1}(\beta_{\,\epsilon}) = \beta_{\,\epsilon},$$
and using again that $A_{\operatorname{inf}}$ is an integral domain, we
therefore conclude that 
$$\varphi^{-1}(\alpha_{\,\epsilon}) = \varphi^{-1}(\xi_{\,\epsilon})
\cdot \alpha_{\,\epsilon},$$
where $\xi_{\,\epsilon} \in A_{\operatorname{inf}}$ is the
element defined by
$$\xi_{\,\epsilon} 
= \varphi(\mu_{\,\epsilon}) \cdot \mu_{\,\epsilon}^{-1} 
= 1 + [\epsilon]
+ [\epsilon]^2 + \dots 
+ [\epsilon]^{p-1}.$$ 
We remark that this element is a generator of the kernel of the ring
homomorphism
$$\xymatrix{
{ A_{\operatorname{inf}} } \ar[r]^-{\theta} &
{ \mathcal{O}_{\Cp}, } \cr
}$$
which, in our formulation, is given by the canonical projection from
the limit that defines the left-hand ring to the term $n = 1$. 

We recall from~\cite[Addendum~5.4.4]{hm4}
and~\cite[Proposition~2.1.4]{h8} that the map of graded
$A_{\operatorname{inf}}$-algebras induced by the cyclotomic 
structure map
$$\xymatrix{
{ \TF_*(\mathcal{O}_{\Cp};p,\Zp) } \ar[r]^-{\gamma} &
{ \TP_*(\mathcal{O}_{\Cp},\Zp) } \cr
}$$
is an isomorphism in non-negative degrees. Hence, it becomes an
isomorphism after inverting the divided Bott element
$\alpha_{\,\epsilon}$. We now define
$$v_{\,\epsilon} \in \TP_{-2}(\mathcal{O}_{\Cp},\Zp)$$
to be the inverse of the image of the divided Bott element
$\alpha_{\,\epsilon}$. It is represented in the Tate spectral sequence
by the element $\smash{ t \in E_{-2,0}^2 }$ defined in
Section~\ref{tatesection}. Clearly, the Galois group $G$ acts on the inverse
divided Bott element by
$$\sigma \cdot v_{\,\epsilon} = v_{\,\sigma(\epsilon)} =
\chi(\sigma)^{-1} \cdot \sigma(\mu_{\,\epsilon}) \cdot
\mu_{\,\epsilon}^{-1} \cdot v_{\,\epsilon}.$$
Moreover, the inverse Frobenius operator $\varphi^{-1}$ on
$\TF_*(\mathcal{O}_{\Cp};p,\Zp)$ gives rise to a meromorphic Frobenius
operator
$$\xymatrix{
{ \TP_*(\mathcal{O}_{\Cp},\Zp) } \ar[r]^-{\varphi} &
{ \TP_*(\mathcal{O}_{\Cp},\Zp), } \cr
}$$
which is defined and invertible away from the divisor
``$\xi_{\,\epsilon} = 0$'' and satisfies
$$\varphi(v_{\,\epsilon}) = \xi_{\,\epsilon} \cdot v_{\,\epsilon}.$$
This completes our recollection of the results from~\cite{h8}.

We consider the ring homomorphism
$$\xymatrix{
{ A_{\operatorname{inf}} } \ar[r]^-{i} &
{ W(k) } \cr
}$$
defined to be the composition
$$\xymatrix{
{ \lim_{n,F}W_n(\mathcal{O}_{\Cp}) } \ar[r] &
{ \lim_{n,F}W_n(k) } \ar[r]^-{\varphi^{\infty}} &
{ \lim_{n,R}W_n(k) } \cr
}$$
of the map induced by the canonical projection of $\mathcal{O}_{\Cp}$
onto its residue field $k$ and the isomorphism $\varphi^{\infty}$ above.
Extension of scalars along $i$ induces an isomorphism 
$$\xymatrix{
{ W(k) \otimes_{A_{\operatorname{inf}}} \TP_*(\mathcal{O}_{\Cp},\Zp) }
\ar[r]^-{i^*} &
{ \TP_*(k,\Zp), } \cr
}$$
the target of which is canonically isomorphic to $\TP_*(k)$. 

\begin{proposition}\label{preferredgenerator}The class
$v = i^*(v_{\,\epsilon}) \in \TP_{-2}(k)$ is independent of the choice
of generator $\epsilon \in \Zp(1)$ and descends to a generator
$v \in \TP_{-2}(\Fp)$. Moreover, the meromorphic Frobenius operator
$$\xymatrix{
{ \TP_*(k) } \ar[r]^-{\varphi} &
{ \TP_*(k) } \cr
}$$
is defined and invertible away from the divisor ``\,$p =0$'' and
satisfies $\varphi(v) = p \cdot v$.
\end{proposition}

\begin{proof}If $\epsilon \in \Zp(1)$ is a fixed generator, then a
general generator is of the form $\sigma(\epsilon) \in \Zp(1)$, for
some $\sigma \in G$. So it suffices to show that $i \colon
A_{\operatorname{inf}} \to W(k)$ maps the unit $u_{\epsilon} =
\chi(\sigma)^{-1} \cdot \sigma(\mu_{\,\epsilon}) \cdot
\mu_{\,\epsilon}^{-1} \in A_{\operatorname{\inf}}$ to $1 \in W(k)$. We
may assume that $\chi(\sigma)$ is an integer, in which case, the image
$u_{\epsilon,n}$ of $u_{\epsilon}$ by the canonical projection from
$A_{\operatorname{inf}}$ to $W_n(\mathcal{O}_{\Cp})$ is given by
$$\begin{aligned}
u_{\epsilon,n}
{} & = \chi(\sigma)^{-1} \cdot ([\sigma(\epsilon^{(n)})]_n - 1) \cdot
([\epsilon^{(n)}]_n - 1)^{-1} \cr
{} & = \chi(\sigma)^{-1} \cdot ([\epsilon^{(n)}]_n^{\chi(\sigma)} - 1)
\cdot ([\epsilon^{(n)}]_n-1)^{-1} \cr
{} & = \chi(\sigma)^{-1} \cdot \, \textstyle{ \sum_{0 \leqslant j <
    \chi(\sigma)} }[\epsilon^{(n)}]_n^j. \cr
\end{aligned}$$
But the canonical projection from $\mathcal{O}_{\Cp}$ to $k$ takes
$\epsilon^{(n)}$ to $1$, and hence, the induced map from
$W_n(\mathcal{O}_{\Cp})$ to $W_n(k)$ takes $u_{\epsilon,n}$ to $1$, as
desired. This shows that the class
$$v = i^*(v_{\,\epsilon}) \in \TP_{-2}(k)$$
is independent of $\epsilon$ as stated. It also shows that this
class is Galois fixed, and hence, that it descends to a class
$v \in \TP_{-2}(\Fp)$. Similarly, we have
$$i(\xi_{\,\epsilon}) = p,$$
which proves the remaining statements.
\end{proof}

\begin{remark}\label{unitremark}It follows from
Proposition~\ref{preferredgenerator} that if we define
$\alpha_n \in \TR_2^n(\Fp;p)$ to be the image of $v^{-1} \in
\TF_2(\Fp;p)$, then $F(\alpha_n) = \alpha_{n-1}$ and
$R(\alpha_n) = p \cdot \alpha_{n-1}$. This shows that,
in~\cite[Proposition~5.4]{hm}, we can arrange for the units
$\lambda_n$ to be equal to $1$.
\end{remark}

\section{The conjugate spectral sequence}\label{conjugatesection}

We recall from the paper~\cite{illusieraynaud} the
higher Cartier isomorphism and the structure of the conjugate spectral
sequence converging to the crystalline cohomology of a scheme smooth
and proper over a perfect field of positive characteristic.

Let $k$ be a perfect field of characteristic $p > 0$ and let
$f \colon X \to \Spec(k)$ be a smooth morphism. The
comparison theorem~\cite[Th\'{e}or\`{e}me~II.1.4]{illusie} gives
a canonical isomorphism of $W_n(k)$-modules
$$\xymatrix{
{ H_{\operatorname{crys}}^i(X/W_n(k)) } \ar[r] &
{ H^i(X,W_n\Omega_X^{\boldsymbol{\cdot}}) } \cr
}$$
from the crystalline cohomology of $X/W_n(k)$ defined
in~\cite{berthelot} to the hypercohomology of $X$ 
with coefficients in the de~Rham-Witt complex. The second spectral
sequence of hypercohomology gives a spectral sequence of
$W_n(k)$-modules
$$E_2^{i,j} = H^i(X,\mathcal{H}^j(W_n\Omega_X^{\boldsymbol{\cdot}}))
\Rightarrow H_{\operatorname{crys}}^{i+j}(X/W_n(k)).$$
Being located in the first quadrant, it converges strongly. If
$f \colon X \to \Spec(k)$ is also proper, and hence of finite
relative dimension $d$, then the $W_n(k)$-modules
$\smash{ E_2^{i,j} }$ are finitely generated
by~\cite[Proposition~II.2.1]{illusie}. Therefore, taking limits with
respect to the restriction maps, we obtain a spectral sequence of
$W(k)$-modules
$$\textstyle{ E_2^{i,j} = \lim_{n,\mathcal{H}^j(R)}
  H^i(X,\mathcal{H}^j(W_n\Omega_X^{\boldsymbol{\cdot}})) \Rightarrow
  H_{\operatorname{crys}}^{i+j}(X/W(k)). }$$
Indeed, by finite generation, the limit systems satisfy the Mittag-Leffler
condition. To rewrite the $E_2$-term we recall the higher Cartier
isomorphism.

We have isomorphisms of sheaves of abelian groups
$$\xymatrix{
{ W_n\Omega_X^j } &
{ W_{2n}\Omega_X^j/(V^nW_n\Omega_X^j + dV^nW_n\Omega_X^{j-1}) }
  \ar[l]_-{R^n} \ar[r]^-{F^n} &
{ \mathcal{H}^j(W_n\Omega_X^{\boldsymbol{\cdot}}). } \cr
}$$
Indeed, the left-hand map is an isomorphism for formal reasons, and
the right-hand map is an  isomorphism
by~\cite[Proposition~III.1.4]{illusieraynaud}. The inverse higher
Cartier operator is defined to be the composite isomorphism
$$C^{-n} = F^n \circ R^{-n}.$$
It follows easily from the definition that the diagram of sheaves of
abelian groups
$$\begin{xy}
(0,7)*+{ W_n\Omega_X^j }="11";
(32,7)*+{ \mathcal{H}^j(W_n\Omega_X^{\boldsymbol{\cdot}}) }="12";
(0,-7)*+{ W_{n-1}\Omega_X^j }="21";
(32,-7)*+{ \mathcal{H}^j(W_{n-1}\Omega_X^{\boldsymbol{\cdot}}) }="22";
{ \ar^-{C^{-n}} "12";"11";};
{ \ar^-{F} "21";"11";};
{ \ar^-{\mathcal{H}^j(R)} "22";"12";};
{ \ar^-{C^{-(n-1)}} "22";"21";};
\end{xy}$$
commutes, and hence, we obtain an induced isomorphism of abelian
groups
$$\xymatrix{
{ \lim_{n,F} H^{-i}(X,W_n\Omega_X^j) } \ar[r] &
{ \lim_{n,\mathcal{H}^j(R)}
  H^{-i}(X,\mathcal{H}^j(W_n\Omega_X^{\boldsymbol{\cdot}})). } \cr 
}$$
In addition, this map is an isomorphism of $W(k)$-modules, provided
that we consider its domain to be a $W(k)$-module via the inverse of
the isomorphism
$$\xymatrix{
{ \lim_{n,F} W_n(k) } \ar[r]^-{\varphi^{\infty}} &
{ \lim_{n,R} W_n(k) = W(k) } \cr
}$$
defined in Section~\ref{dividedbottsection}.
We record these recollections from~\cite{illusie,illusieraynaud}
in the following result.

\begin{theorem}\label{conjugate}If $k$ is a perfect field of
characteristic $p > 0$ and $f \colon X \to \Spec(k)$ a smooth
and proper morphism, then there is a spectral sequence of
$\,W(k)$-modules
$$\textstyle{ E_2^{i,j} = \lim_{n,F} H^i(X,W_n\Omega_X^j) \Rightarrow
H_{\operatorname{crys}}^{i+j}(X/W(k)) }$$
which converges strongly.
\end{theorem}

We call the spectral sequence in Theorem~\ref{conjugate} the
\emph{conjugate} spectral sequence. The analysis of its
$E_2$-term in~\cite[Th\'{e}or\`{e}me~III.2.2]{illusieraynaud} gives
the following result.

\begin{corollary}\label{conjugatecorollary}Let $k$ be a finite field of order
$q =p^r$, let $W$ be its ring of $p$-typical Witt vectors, let
$\iota \colon W \to \C$ be an embedding, and let $f \colon X \to
\Spec(k)$ be a smooth and proper morphism. In this situation, there is
a strongly convergent spectral sequence of degree-wise finite
dimensional $\C$-vector spaces
$$\textstyle{ E_2^{i,j} = (\lim_{n,F} H^i(X,W_n\Omega_X^j))
  \otimes_{W,\iota} \C \Rightarrow H_{\operatorname{crys}}^{i+j}(X/W)
  \otimes_{W,\iota} \C }.$$
\end{corollary}

\begin{proof}The spectral sequence in the statement is obtained from
the conjugate spectral sequence by extension of scalars along the 
ring homomorphism $\iota \colon W \to \C$, which is flat. Finally,
by~\cite[Th\'{e}or\`{e}me~III.2.2]{illusieraynaud}, each of the
$W(k)$-modules $\smash{ E_2^{i,j} }$ in the conjugate spectral
sequence is finitely generated modulo torsion.
\end{proof}

\begin{remark}\label{conjugateremark}In the situation of
Theorem~\ref{conjugate},
\cite[Th\'{e}or\`{e}me~III.2.2]{illusieraynaud} also shows that the
torsion sub-$W(k)$-module of $\lim_{n,F}H^i(X,W_n\Omega_X^j)$
is annihilated by $p$. However, this torsion sub-$W(k)$-module may
well not be finitely generated,\footnote{\,The weight spectral
  sequence of~\cite{bhattmorrowscholze1} does not have this
  pathology.}
as the example of a supersingular K3 surface shows;
see~\cite[IV.2.15.6]{illusieraynaud} and~\cite[II.7.2]{illusie}.
\end{remark}

\section{The Hodge spectral sequence}\label{hodgesection}

In this section, we construct the Hodge spectral sequence for periodic
topological cyclic homology of schemes smooth and proper over a
perfect field of characteristic $p > 0$.

In general, let $X$ be a site, let $F$ be a presheaf of spectra on
$X$, and let
$$\xymatrix{
{ F(U) } \ar[r]^-{\eta_U} &
{ H^{\boldsymbol{\cdot}}(U,F) } \cr
}$$
be the canonical morphism to the associated cohomology presheaf. The
latter may be defined by means of the Godement construction, if
the site $X$ has enough points, and by a fibrant replacement in an
appropriate model structure, in general;
see~\cite{thomason,jardine}. In either case, there is a
hypercohomology spectral sequence
$$E_{i,j}^2 = H^{-i}(U,\pi_j(F)\tilde{\phantom{x}}) \Rightarrow
H^{-i-j}(U,F)$$
from the cohomology of $U$ with coefficients in the sheaves
assoaciated with the presheaves of homotopy groups which converges
conditionally to the homotopy groups of the global sections of the
cohomology presheaf. If the morphism $\eta_U$ is a weak equivalence
for every object $U$ of $X$, then we say that the presheaf $F$
satisfies descent. In this case, the abutment of the hypercohomlogy
spectral sequence is canonically identified with $\pi_{i+j}(F(U))$.

If $X$ is a scheme, then by~\cite[Definition~3.2.3]{gh}
or~\cite[Theorem~1.3]{blumbergmandell} the presheaf of spectra on the
small \'{e}tale site of $X$ that to an \'{e}tale morphism $h \colon U
\to X$ assigns $\THH(U)$ satisfies descent. Hence, in this situation
the hypercohomology spectral sequence takes the form
$$E_{i,j}^2 = H^{-i}(U,\THH_j(\mathcal{O}_X)) \Rightarrow
\THH_{i+j}(U),$$
where, as in Section~\ref{thhsection}, $\THH_j(\mathcal{O}_X)$ is the
sheaf of $\mathcal{O}_X$-modules on the small \'{e}tale site of $X$
associated with the presheaf that to an \'{e}tale morphism $h \colon U
\to X$ assigns $\THH_j(U)$. We recall
from~\cite[Proposition~3.2.1]{gh} that this $\mathcal{O}_X$-module is
quasi-coherent. In particular, in the case where $U$ is affine, the
spectral sequence collapses and the edge homomorphism
$$\xymatrix{
{ \THH_j(U) } \ar[r] &
{ H^0(U,\THH_j(\mathcal{O}_X)) } \cr
}$$
is an isomorphism.

We next consider the presheaf of spectra on the small \'{e}tale site
of $X$ that to an \'{e}tale morphism $h \colon U \to X$ assigns the
Tate spectrum
$$F(U) = \hat{H}^{\boldsymbol{\cdot}}(C_{p^n},\THH(U)),$$
and let $\hat{H}^{-j}(C_{p^n},\THH(\mathcal{O}_X))$ be the sheaf
of $W_n(\mathcal{O}_X)$-modules on the small \'{e}tale site of $X$
that to an \'{e}tale morphism $h \colon U \to X$ assigns
$\hat{H}^{-j}(C_{p^n},\THH(U))$.

\begin{lemma}\label{finitefiltration}Let $k$ be a perfect field of
characteristic $p > 0$. If $f \colon k \to A$ is a smooth morphism of
relative dimension $d$, then the Tate spectral sequence
$$E^2 = \hat{H}^{-i}(C_{p^n},\THH_j(A)) \Rightarrow
\hat{H}^{-i-j}(C_{p^n},\THH(A))$$
converges strongly and $E_{i,j}^{\infty}$ vanishes for $j \geqslant
d+2n$. In particular, the induced filtration of the abutment is of
finite length.
\end{lemma}

\begin{proof}The spectral sequence in question takes the form
$$E^2 = \Lambda\{u_n\} \otimes S\{t^{\pm1},\alpha\} \otimes \Omega_A^*
\Rightarrow \hat{H}^{-*}(C_{p^n},\THH(A)),$$
where the classes $u_n$, $t$, and $\alpha$ have bidegrees $(-1,0)$,
$(-2,0)$, and $(0,2)$, and where $\Omega_A^*$ is located in bidegrees
$(0,j)$ with $0 \leqslant j \leqslant d$. By~\cite[Lemma~5.4]{hm}, the
classes $t$ and $\alpha$ are infinite cycles, and if $A = k$, then 
$$d^{2n+1}(u_n) = \lambda \cdot t^{n+1}\alpha^n$$
for some unit $\lambda \in \Fp$. By comparison, the
$\mathbb{T}$-Tate spectral sequence takes the form
$$E^2 = S\{t^{\pm1},\alpha\} \otimes \Omega_A^* \Rightarrow
\hat{H}^{-*}(\mathbb{T},\THH(A)),$$
and the inclusion of $C_{p^n}$ in $\mathbb{T}$ induces a map from the
latter spectral sequence to the form, which, on $E^2$-terms,
is induced by the unit map $\eta \colon k \to \Lambda\{u_n\}$. Now, for
degree reasons, all nonzero differentials in the latter spectral
sequence are even. Therefore, in the former spectral sequence, the
$d^{2n+1}$-differential is given by
$$d^{2n+1}(u_n \otimes \omega) = \lambda \cdot t^{n+1}\alpha^n\omega.$$
This shows that $E_{i,j}^{2n+2}$ vanishes for $j  \geqslant d+2n$, and
hence, so does $E_{i,j}^{\infty}$. Finally, since the Tate spectral
sequence always converges conditionally, we conclude that, in the
situation at hand, it converges strongly;
see~\cite[Theorem~7.1]{boardman}.
\end{proof}

\begin{proposition}\label{tateaffinequasicoherent}Let $k$ be a perfect
field of characteristic $p > 0$, let $n$ be a positive integer, and
let $f \colon k \to A$ be a smooth morphism. If $h \colon A \to B$ is
an \'{e}tale morphism, then the induced map of graded
$W_n(B)$-algebras
$$\xymatrix{
{ W_n(B) \otimes_{W_n(A)} \hat{H}^{-*}(C_{p^n},\THH(A)) } \ar[r] &
{ \hat{H}^{-*}(C_{p^n},\THH(B)) } \cr
}$$
is an isomorphism.
\end{proposition}

\begin{proof}The Tate spectral sequence is a spectral sequence of
$W_{n+1}(A)$-modules
$$E_{i,j}^2 = \hat{H}^{-i}(C_{p^n},F_*^n\THH_j(A)) \Rightarrow
R_*\hat{H}^{-i-j}(C_{p^n},\THH(A)),$$
where $F_*^n(-)$ and $R_*(-)$ indicates the restriction of scalars along 
the Frobenius $F^n \colon W_{n+1}(A) \to A$ and restriction
$R \colon W_{n+1}(A) \to W_n(A)$, respectively. We recall
from~\cite[Theorem~B]{borger} that the morphism
$W_{n+1}(h) \colon W_{n+1}(A) \to W_{n+1}(B)$ again is \'{e}tale and
hence flat. Therefore, by extension of scalars along this morphism, we
obtain a spectral sequence of $W_{n+1}(B)$-modules
$$\begin{aligned}
E_{i,j}^2 
{} & = \hat{H}^{-i}(C_{p^n}, W_{n+1}(B) \otimes_{W_{n+1}(A)}
F^n_*\THH_j(A)) \cr
{} & \Rightarrow W_{n+1}(B) \otimes_{W_{n+1}(A)}
R_*\hat{H}^{-i-j}(C_{p^n},\THH(A)), \cr
\end{aligned}$$
and Lemma~\ref{finitefiltration} shows that both this spectral
sequence and the spectral sequence
$$E_{i,j}^2 = \hat{H}^{-i}(C_{p^n},F_*^n\THH_j(B)) \Rightarrow
R_*\hat{H}^{-i-j}(C_{p^n},\THH(B)),$$
converge strongly. The morphism $h \colon A \to B$ induces a map
of spectral sequences from the former to the latter, and we claim that
this map is an isomorphism. To prove this, it suffices to show that the
induced map of $W_{n+1}(B)$-modules
$$\xymatrix{
{ W_{n+1}(B) \otimes_{W_{n+1}(A)} F_*^n\THH_j(A) } \ar[r] &
{ F_*^n\THH_j(B) } \cr
}$$
is an isomorphism. If $n = 0$, then this is proved
in~\cite[Proposition~3.2.1]{gh}, and the general case follows from
this case and from the fact proved in~\cite[Corollary~15.4]{borger1}
that the left-hand diagram below is a cocartesian square of
commutative rings.
$$\begin{xy}
(-40,7)*+{ W_{n+1}(A) }="11";
(-13,7)*+{ W_{n+1}(B) }="12";
(13,7)*+{ W_{n+1}(A) }="13";
(40,7)*+{ W_{n+1}(B) }="14";
(-40,-7)*+{ W_n(A) }="21";
(-13,-7)*+{ W_n(B) }="22";
(13,-7)*+{ W_n(A) }="23";
(40,-7)*+{ W_n(B) }="24";
{ \ar^-{W_{n+1}(h)} "12";"11";};
{ \ar^-{F} "21";"11";};
{ \ar^-{F} "22";"12";};
{ \ar^-{W_n(h)} "22";"21";};
{ \ar^-{W_{n+1}(h)} "14";"13";};
{ \ar^-{R} "23";"13";};
{ \ar^-{R} "24";"14";};
{ \ar^-{W_n(h)} "24";"23";};
\end{xy}$$
We therefore conclude that the map of graded $W_{n+1}(B)$-modules
$$\xymatrix{
{ W_{n+1}(B) \otimes_{W_{n+1}(A)}R_*\hat{H}^{-*}(C_{p^n},\THH(A)) }
\ar[r] &
{ R_*\hat{H}^{-*}(C_{p^n},\THH(B)) } \cr
}$$
induced by $h \colon A \to B$ is an isomorphism. Finally, by
loc.~cit., the right-hand diagram above also is a cocartesian square
of commutative rings. This shows that the map in the statement is an
isomorphism as desired.
\end{proof}

\begin{corollary}\label{tatequasicoherent}Let $k$ be a perfect field
of characteristic $p > 0$, let $n$ be a positive integer, and let $f
\colon X \to \Spec(k)$ be a smooth morphism of finite relative
dimension. In this situation, the graded $W_n(\mathcal{O}_X)$-algebra
$\hat{H}^{-*}(C_{p^n},\THH(\mathcal{O}_X))$ is quasi-coherent.
\end{corollary}

\begin{proof}As proved
in~\cite[Expos\'{e}~VII,~Proposition~3.1]{SGA4II}, the \'{e}tale
morphisms $h \colon A \to B$ with $\Spec(A)$ open in $X$ generate the
small \'{e}tale topos of $X$. Therefore, the statement follows from
Proposition~\ref{tateaffinequasicoherent}.
\end{proof}

\begin{proposition}\label{tatedescent}Let $k$ be a perfect field of
characteristic $p > 0$, let $n$ be a positive integer, and let $f
\colon X \to \Spec(k)$ is a smooth morphism of finite relative
dimension. In this situation, the presheaf of spectra on the small \'{e}tale
site of $X$ that to an \'{e}tale morphism $h \colon U \to X$ assigns
$\hat{H}^{\boldsymbol{\cdot}}(C_{p^n},\THH(U))$ satisfies descent.
\end{proposition}

\begin{proof}Let again $F$ denote the presheaf in question. We wish to
prove that for every \'{e}tale morphism $h \colon U \to X$, the
canonical morphism
$$\xymatrix{
{ F(U) } \ar[r]^-{\eta_U} &
{ H^{\boldsymbol{\cdot}}(U,F) } \cr
}$$
is a weak equivalence. Corollary~\ref{tatequasicoherent} shows that,
for $U$ affine, the sheafification map
$$\xymatrix{
{ \pi_j(F(U)) } \ar[r] &
{ H^0(U,\pi_j(F)\tilde{\phantom{x}}) } \cr
}$$
is an isomorphism and that the descent spectral sequence
$$E_{i,j}^2 = H^{-i}(X,\pi_j(F)\tilde{\phantom{x}}) \Rightarrow
H^{-i-j}(U,F)$$
collapses. But the sheafification map is equal to the composition of
the map of homotopy groups induced by $\eta_U$ and the edge
homomorphism of the descent spectral sequence, so we conclude that
$\eta_U$ is a weak equivalence in this case. If $h \colon U \to X$ is
a general \'{e}tale morphism, then we choose a hyper-covering
$\mathcal{U}$ of $U$ by affine schemes \'{e}tale over $U$ and
consider the diagram in which the two vertical morphisms are the
canonical morphisms to the respective \v{C}ech cohomology spectra,
$$\begin{xy}
(0,7)*+{ F(U) }="11";
(40,7)*+{ H^{\boldsymbol{\cdot}}(U,F) }="12";
(0,-7)*+{ \check{H}^{\boldsymbol{\cdot}}(\mathcal{U},F) }="21";
(40,-7)*+{ \check{H}^{\boldsymbol{\cdot}}(\mathcal{U},
  H^{\boldsymbol{\cdot}}(-,F)). }="22";
{ \ar^-{\eta_U} "12";"11";};
{ \ar^-{\check{\eta}_{\mathcal{U}}} "21";"11";};
{ \ar^-{\check{\eta}_{\mathcal{U}}} "22";"12";};
{ \ar^-{\check{H}^{\boldsymbol{\cdot}}(\mathcal{U},\eta)} "22";"21";};
\end{xy}$$
The lower horizontal morphism
is a weak equivalence by the case already considered, and the
right-hand vertical morphism is a weak equivalence, since sheaf
cohomology satisfies descent. We claim that also the left-hand vertical
morphism is a weak equivalence. We write this morphism as the
composition
$$\xymatrix{
{ \hat{H}^{\boldsymbol{\cdot}}(C_{p^n},\THH(U)) } \ar[r] &
{ \hat{H}^{\boldsymbol{\cdot}}(C_{p^n},
\check{H}^{\boldsymbol{\cdot}}(\mathcal{U},\THH(-))) } \ar[r] &
{ \check{H}^{\boldsymbol{\cdot}}(\mathcal{U},
\hat{H}^{\boldsymbol{\cdot}}(C_{p^n},\THH(-))) } \cr
}$$
of the morphism in Tate cohomlogy induced by the corresponding
morphism in topological Hochschild homology and the canonical map from
Tate cohomology of a homotopy limit to the homotopy limit of the Tate
cohomology of the individual terms in the limit system.
The former map is a weak equivalence, since topological Hochschild
homology satisfies descent, and the latter map is a weak equivalence,
since the spectra in the limit system are universally bounded
below by the relative dimension of $f \colon X \to \Spec(k)$.
\end{proof}

\begin{remark}\label{openquestion}We do not know whether or not
Corollary~\ref{tatequasicoherent} and Proposition~\ref{tatedescent} 
hold for all schemes $X$.
\end{remark}

We view the next result as an analogue of analytic continuation.

\begin{proposition}\label{continuation}Let $k$ be a perfect field of
characteristic $p > 0$, let $n \geqslant 1$ be an integer, and let
$f \colon X \to \Spec(k)$ is a smooth morphism of relative dimension
$d$. In this situation, the map of graded $W_n(\mathcal{O}_X)$-algebras
$$\xymatrix{
{ \TR_i^n(\mathcal{O}_X;p) } \ar[r]^-{\gamma} &
{ \hat{H}^{-i}(C_{p^n},\THH(\mathcal{O}_X)) } \cr
}$$
induced by the cyclotomic structure morphism is an isomorphism
for $i \geqslant d$.
\end{proposition}

\begin{proof}The domain and target $W_n(\mathcal{O}_X)$-modules of the
map in question both are quasi-coherent. Therefore, it suffices to
take $X = \mathbb{A}_k^d$ and show that the induced map of global
sections is an isomorphism for $i \geqslant d$. Moreover,
by~\cite[Theorem~2.4]{tsalidis}, it suffices to consider the case
$n = 1$. Now, in the case $d = 0$, the result was proved
in~\cite[Proposition~5.3]{hm}; see also loc.~cit., Remark~5.5. In
general, we have an equivalence of $C_p$-spectra
$$\xymatrix{
{ \THH(k) \otimes N^{\cy}(\langle x_1,\dots,x_d \rangle) }
\ar[r]^-{\alpha} &
{ \THH(k[x_1,\dots,x_d]) } \cr
}$$
where the second factor on the left-hand side is the cyclic
bar-construction of the free commutative monoid on the indicated
generators. We refer to~\cite[Proposition~3]{h3} for the definition of
the map $\alpha$ and for the proof that it is a weak
equivalence. Moreover, there is a commutative diagram of spectra
$$\xymatrix{
{ \THH(k) \otimes N^{\cy}(\langle x_1,\dots,x_d \rangle) }
\ar[r]^-{\alpha} \ar[d]^-{\gamma \otimes \id} &
{ \THH(k[x_1,\dots,x_d]) } \ar[d]^-{\gamma} \cr
{ \hat{H}^{\boldsymbol{\cdot}}(C_p,\THH(k)) \otimes N^{\cy}(\langle
  x_1,\dots,x_d \rangle) } \ar[r]^-{\hat{\alpha}} &
{ \hat{H}^{\boldsymbol{\cdot}}(C_p,\THH(k[x_1,\dots,x_d])) } \cr
}$$
and the lower horizontal map also is a weak equivalence;
see~\cite[Corollary~9.1]{hm}. We claim that the reduced singular
homology of $N^{\cy}(\langle x_1,\dots,x_d \rangle)$ is concentrated
in the degrees $0 \leqslant i \leqslant d$. Granting this, we conclude
from the case $d = 0$ that the left-hand vertical map in the diagram
above induces an isomorphism of homotopy groups in degrees greater
than or equal to $d$. But then the same holds for the right-hand
vertical map, which proves the proposition. To prove the claim we use
that there are canonical isomorphisms
$$\xymatrix{
{ \Omega_{\Z[x_1,\dots,x_d]/\Z}^i } \ar[r] &
{ \HH_i(\Z[x_1,\dots,x_d]/\Z) } \ar[r] &
{ \tilde{H}_i(N^{\cy}(\langle x_1,\dots,x_d \rangle),\Z) } \cr
}$$
and the easy calculation of the left-hand side.
\end{proof}

\begin{theorem}\label{tatesheaf}Let $k$ be a perfect field
of characteristic $p > 0$, let $n$ be a positive integer, and let $f
\colon X \to \Spec(k)$ be a smooth morphism of finite relative
dimension. In this situation, the canonical map of graded
$W_n(\mathcal{O}_X)$-algebras
$$\xymatrix{
{ W_n\Omega_X^* \otimes_{f^*W_n(k)} f^*\hat{H}^{-*}(C_{p^n},\THH(k)) }
\ar[r]^-{\hat{\eta}} &
{ \hat{H}^{-*}(C_{p^n},\THH(\mathcal{O}_X)) } \cr
}$$
is an isomorphism.
\end{theorem}

\begin{proof}We consider the commutative diagram of graded
$W_n(\mathcal{O}_X)$-algebras in which the horizontal maps are the
canonical maps and the vertical maps are induced by the cyclotomic
structure maps,
$$\xymatrix{
{ W_n\Omega_X^* \otimes_{f^*W_n(k)} f^*\TR_*^n(k;p) } \ar[r]^-{\eta}
\ar[d]^-{\id \otimes \gamma} &
{ \TR_*^n(\mathcal{O}_X;p) } \ar[d]^-{\gamma} \cr
{ W_n\Omega_X^* \otimes_{f^*W_n(k)} f^*\hat{H}^{-*}(C_{p^n},\THH(k)) }
\ar[r]^-{\hat{\eta}} & 
{ \hat{H}^{-*}(C_{p^n},\THH(\mathcal{O}_X)). } \cr
}$$
We recall from~\cite[Theorem~B]{h} that the top horizontal map is an
isomorphism and from Proposition~\ref{continuation} that the
right-hand vertical map is an isomorphism in degrees greater than or
equal to the relative dimension of $f \colon X \to \Spec(k)$. Since the
graded $W_n(\mathcal{O}_X)$-algebras in the bottom row both are
2-periodic, the theorem follows. 
\end{proof}


\begin{theorem}\label{hodge}If $X$ is a scheme smooth and proper
morphism over a perfect field $k$ of characteristic $p > 0$, then
there is a spectral sequence of $\,W(k)$-modules 
$$\textstyle{ E_{i,j}^2 = \bigoplus_{m \in \Z} \lim_{n,F}
H^{-i}(X,W_n\Omega_X^{j+2m}) \Rightarrow
\TP_{i+j}(X) }$$
which converges strongly.
\end{theorem}

\begin{proof}We first argue that the canonical morphism
$$\xymatrix{
{ \TP(X) = \hat{H}^{\boldsymbol{\cdot}}(\mathbb{T},\THH(X)) } \ar[r] &
{ \holim_{n,F} \hat{H}^{\boldsymbol{\cdot}}(C_{p^n},\THH(X)) } \cr
}$$
is a weak equivalence. It follows from~\cite[Lemma~2.1.1]{h8} that
this morphism becomes a weak equivalence after $p$-completion, and
therefore, it suffices to show that the domain and target are
$p$-complete. This, in turn, follows from the definition of
the Tate spectrum, once we shows that $\THH(X)$ is $p$-complete and
bounded below. That the spectrum $\THH(X)$ is $p$-complete is an
immediate consequence of its homotopy groups being $k$-vector spaces,
and hence, annihilated by $p$. To see that it is bounded below, we
consider the conditionally convergent descent spectral sequence of
$k$-vector spaces
$$E_{i,j}^2 = H^{-i}(X,\THH_j(\mathcal{O}_X)) \Rightarrow
\THH_{i+j}(X).$$
Since the structure morphism $f \colon X \to \Spec(k)$ is smooth and
proper, its relative dimension $d$ is finite. Therefore, the
$k$-vector space $\smash{ E_{i,j}^2 }$ is zero, unless
$-d \leqslant i \leqslant 0$ and $j \geqslant 0$, which shows that the
spectral sequence converges strongly and that $\THH_j(X)$ is zero for
$j < -d$. So the canonical morphism at the beginning of the proof is a
weak equivalence as desired.

Next, by Proposition~\ref{tatedescent}, we have descent spectral
sequences
$$E_{i,j}^2 = H^{-i}(X,\hat{H}^{-j}(C_{p^n},\THH(\mathcal{O}_X)))
\Rightarrow \hat{H}^{-i-j}(C_{p^n},\THH(X)),$$
which converge strongly. Indeed, the spectral
sequences converge conditionally, and $\smash{ E_{i,j}^2 }$ is zero,
unless $-d \leqslant i \leqslant 0$. Moreover, identifying
$$\hat{H}^{-*}(C_{p^n},\THH(k)) = S_{W_n(k)}\{v^{\pm1}\}$$
with $v \in \TP_{-2}(k)$, Theorem~\ref{tatesheaf} gives canonical
isomorphisms of $W_n(\mathcal{O}_X)$-modules
$$\xymatrix{
{ \bigoplus_{m \in \Z} W_n\Omega_X^{j+2m} } \ar[r] &
{ \hat{H}^{-j}(C_{p^n},\THH(\mathcal{O}_X)), } \cr
}$$
and hence, we may rewrite the descent spectral sequences above in the
form
$$\textstyle{ E_{i,j}^2 = \bigoplus_{m \in \Z}
  H^{-i}(X,W_n\Omega_X^{j+2m}) 
\Rightarrow \hat{H}^{-i-j}(C_{p^n},\THH(X)). }$$
We claim that the $W_n(k)$-module $\smash{ E_{i,j}^2 }$ is finitely generated
for all values of $(i,j)$. Indeed, since $f \colon X
\to \Spec(k)$ is smooth and proper of relative dimension $d$, it
follows from~\cite[Proposition~II.2.1]{illusie} that the
$W_n(k)$-module $\smash{ H^{-i}(X,W_n\Omega_X^{j+2m}) }$ is finitely
generated, for all integers $m$. But unless $0 \leqslant j+2m
\leqslant d$, said $W_n(k)$-module is zero, so the claim
follows. Moreover, for cohomological dimension reasons, $\smash{
  E_{i,j}^2 }$ is zero, unless $-d \leqslant i \leqslant 0$, so
we conclude that also $\smash{ \hat{H}^{-i}(C_{p^n},\THH(X)) 
}$ is a finitely generated  $W_n(k)$-module, for all integers
$i$. Hence, taking limits with respect to the Frobenius maps, we
obtain a spectral sequence of $W(k)$-modules
$$\textstyle{ E_{i,j}^2 = \bigoplus_{m \in \Z} \lim_{n,F}
H^{-i}(X,W_n\Omega_X^{j+2m}) \Rightarrow \lim_{n,F}
\hat{H}^{-i-j}(C_{p^n},\THH(X)), }$$
and the corresponding higher derived limits vanish. Indeed, by finite
generation, the limit systems satisfy the Mittag-Leffler
condition. This also implies that the canonical map
$$\xymatrix{
{ \TP_i(X) } \ar[r] &
{ \lim_{n,F} \hat{H}^{-i}(C_{p^n},\THH(X)) } \cr
}$$
is an isomorphism, so the theorem follows.
\end{proof}

\begin{corollary}\label{hodgecorollary}Let $k$ be a finite field of order
$q =p^r$, let $W$ be its ring of $p$-typical Witt vectors, and let
$\iota \colon W \to \C$ be a choice of embedding. If $X$ is a scheme
smooth and proper over $k$, then there is a strongly convergent
spectral sequence of degree-wise finite dimensional $\C$-vector spaces
$$\textstyle{ E_{i,j}^2 = \bigoplus_{m \in \Z} (\lim_{n,F}
H^{-i}(X,W_n\Omega_X^{j+2m})) \otimes_{W,\iota} \C \Rightarrow
\TP_{i+j}(X) \otimes_{W,\iota} \C }.$$
\end{corollary}

\begin{proof}We have the strongly convergent spectral sequence of
$W$-modules
$$\textstyle{ E_{i,j}^2 = \bigoplus_{m \in \Z} \lim_{n,F}
H^{-i}(X,W_n\Omega_X^{j+2m}) \Rightarrow
\TP_{i+j}(X) }$$
from Theorem~\ref{hodge}. For cohomological dimension
reasons, it is concentrated in the strip $-d \leqslant i \leqslant 0$,
where $d$ is the relative dimension of $X$ over $k$. Therefore, we obtain the desired spectral sequence by extension of
scalars along the flat ring homomorphism $\iota \colon W \to \C$. Finally,
by~\cite[Th\'{e}or\`{e}me~III.2.2]{illusieraynaud}, we see that
the $\C$-vector spaces $\smash{ E_{i,j}^2 }$ are finite dimensional as
in the proof of Corollary~\ref{conjugatecorollary}.
\end{proof}

\section{Regularized determinants}\label{regularizedsection}

In this section, we recall the definitions of regularized determinant
and anomalous dimension following Deninger~\cite{deninger} and
complete the proof of Theorem~\ref{main}.

The definition requires a choice of principal branch
$\operatorname{Arg}(\lambda)$ of the argument of $\lambda \in \C^*$.
We follow~\cite[Section~1]{deninger} and choose
$-\pi < \operatorname{Arg}(\lambda) \leqslant \pi$ and define
$$\lambda^{-s} = \lvert \lambda \rvert^{-s}
e^{-is\operatorname{Arg}(\lambda)}$$
for $\lambda \in \C^*$ and $s \in \C$. Let $V$ be a $\C$-vector
space\footnote{\,There is no topology on $V$; in particular, it is not
  a Hilbert space.} of at most countably infinite dimension and let
$\Theta \colon V \to V$ be a $\C$-linear map such that:
\begin{enumerate}
\item[(1)]For every $\lambda \in \C$, the generalized eigenspace
$V_{\lambda}$ of $\Theta \colon V \to V$ associated to the eigenvalue
$\lambda$ is finite dimensional.
\item[(2)]The Dirichlet series $\textstyle{ \sum_{\lambda \in \C^*}
    \dim_{\C}(V_{\lambda})\lambda^{-s} }$
converges absolutely for $\operatorname{Re}(s) \gg 0$ and admits a
meromorphic continuation $\zeta_{\Theta}(s)$ to the halfplane
$\operatorname{Re}(s) > -\epsilon$ for some $\epsilon > 0$ which is
holomorphic at $s = 0$.
\end{enumerate}
The anomalous dimension and regularized determinant of $\Theta \colon
V \to V$ are now defined to be the complex numbers
$$\begin{aligned}
\dim_{\infty}(\Theta \mid V)
{} & = \; \dim_{\C}(V_0) + \zeta_{\bar{\Theta}}(0), \cr
\textstyle{ \det_{\infty}(\Theta \mid V)\, } 
{} & = \, \begin{cases}
e^{-\zeta_{\Theta}'(0)} & \text{if $\dim_{\C}(V_0) = 0$,} \cr
0 & \text{if $\dim_{\C}(V_0) > 0$,} \cr
\end{cases} \cr
\end{aligned}$$
where, in the top line, $\bar{\Theta} \colon V/V_0 \to V/V_0$ is the
map induced by $\Theta \colon V \to V$. If $V$ is finite dimensional,
then $\dim_{\infty}(\Theta \mid V) = \dim_{\C}(V)$ and
$\det_{\infty}(\Theta \mid V) = \det(\Theta \mid V)$. Moreover, if
$\delta$ is a positive real number, then one has
$$\textstyle{ \det_{\infty}(\delta \cdot \Theta \mid V) =
  \delta^{\dim_{\infty}(\Theta \mid V)}\det_{\infty}(\Theta \mid V), }$$
as expected. More importantly, by~\cite[Lemma~1.2]{deninger}, given a
commutative diagram
$$\xymatrix{
{ 0 } \ar[r] &
{ V' } \ar[r] \ar[d]^-{\Theta'} &
{ V } \ar[r] \ar[d]^-{\Theta} &
{ V'' } \ar[r] \ar[d]^-{\Theta''} &
{ 0 } \cr 
{ 0 } \ar[r] &
{ V' } \ar[r] &
{ V } \ar[r] &
{ V'' } \ar[r] &
{ 0 } \cr
}$$
of $\C$-vector spaces and $\C$-linear maps with exact rows such that
the right- and left-hand vertical maps satisfy~(1)--(2),
then so does the middle vertical map and
$$\begin{aligned}
\dim_{\infty}(\Theta \mid V) 
{} & = \dim_{\infty}(\Theta' \mid V') + \dim_{\infty}(\Theta'' \mid
V'') \cr
\textstyle{ \det_{\infty}(\Theta \mid V) }
{} & = \textstyle{ \; \det_{\infty}(\Theta' \mid V) \;\, \cdot \,\; \det_{\infty}(\Theta''
\mid V''). } \cr
\end{aligned}$$
The following result is a special case
of~\cite[Corollary~2.8]{deninger}.

\begin{proposition}\label{deningerformula}Let $T$ be an
anticommutative graded $\,\C$-algebra such that the sub-$\C$-vector
space $T_j \subset T$ of homogeneous elements of degree $j$ is finite
dimensional for all integers $j$. Let $\Theta \colon T \to T$ be a
graded $\,\C$-linear derivation and suppose that there exists a unit
$v \in  T_{-2}$ such that
$\Theta(v) = \frac{2\pi i}{\log q} \cdot v$. In this situation,
$$\begin{aligned}
\textstyle{ \det_{\infty}(s \cdot \id - \,\Theta \mid T_{2*+ j}) }
{} & = \det(\id - q^{-s \cdot \Theta} \mid T_j) , \cr
\dim_{\infty}(s \cdot \id -\,\Theta \mid T_{2*+j})
{} & = \, 0. \cr
\end{aligned}$$
for all $s \in \C$ and for all integers $j$.
\end{proposition}

\begin{remark}The two side of the equality in the statement of
Proposition~\ref{deningerformula} only depend on the parity of the
integer $j$. Indeed, we have the exact sequence
$$\xymatrix{
{ 0 } \ar[r] &
{ \frac{2 \pi i}{\log q}\Z } \ar[r] &
{ \C } \ar[r] &
{ \C^* } \ar[r] &
{ 0, } \cr
}$$
where the right-hand map takes $\alpha \in \C$ to
$q^{\alpha} \in \C^*$. In its most basic form, Deninger's formula
expresses the inverse non-archimedean Euler factor $1 - \lambda
q^{-s}$ with $\lambda \in \C^*$ as the regularized product of the
$s- \alpha$ with $\alpha \in \C$ ranging over all solutions to
$q^{\alpha} = \lambda$.
\end{remark}

\begin{proof}[Proof of Theorem~\ref{main}]By applying
Proposition~\ref{deningerformula} to the anticommutative graded
$\C$-algebra $\TP_*(X) \otimes_{W,\iota} \C$ and a graded
$\C$-linear derivation
$$\xymatrix{
{ \TP_*(X) \otimes_{W,\iota} \C } \ar[r]^-{\Theta} &
{ \TP_*(X) \otimes_{W,\iota} \C } \cr
}$$
as in the statement of the theorem, we find that for all $s \in \C$
and $j \in \Z$,
$$\textstyle{
\det_{\infty}(s \cdot \id - \,\Theta \mid \TP_{2*+j}(X)
\otimes_{W,\iota} \C) =\det(\id - q^{-s}\Fr_q^* \mid \TP_j(X)
\otimes_{W,\iota} \C). }$$
Moreover, by comparing the Hodge spectral sequence in
Corollary~\ref{hodgecorollary} and the conjugate spectral sequence in
Corollary~\ref{conjugatecorollary} and using that the determinant is
multiplicative on exact sequence, we furthermore conclude that
$$\det(\id - q^{-s}\Fr_q^* \mid \TP_j(X)
\otimes_{W,\iota} \C) = \begin{cases}
\det(\id - q^{-s}\Fr_q^* \mid
H_{\operatorname{crys}}^{\operatorname{od}}(X/W) \otimes_{W,\iota}
\C) & \cr
\det(\id - q^{-s}\Fr_q^* \mid
H_{\operatorname{crys}}^{\operatorname{ev}}(X/W) \otimes_{W,\iota}
\C) & \cr
\end{cases}$$
respectively, as $j$ is odd or even. Putting the two equalities
together, we find that
$$\frac{ \det_{\infty}(s \cdot \id - \,\Theta \mid
\TP_{\operatorname{od}}(X) \otimes_{W,\iota} \C) }{
\det_{\infty}(s \cdot \id - \,\Theta \mid \TP_{\operatorname{ev}}(X)
\otimes_{W,\iota} \C) }  = \frac{ \det(\id - q^{-s}  \Fr_q^* \mid
  H_{\operatorname{crys}}^{\operatorname{od}}(X/W)
  \otimes_{W,\iota}\C) }{ \det(\id - q^{-s}  \Fr_q^* \mid
  H_{\operatorname{crys}}^{\operatorname{ev}}(X/W)
  \otimes_{W,\iota}\C) },$$
and finally, it is proved
in~\cite[Th\'{e}or\`{e}me~VII.3.2.3]{berthelot} that the right-hand
side is equal to the Hasse-Weil zeta function $\zeta(X,s)$.
\end{proof}

By the analogue of analytic continuation,
Proposition~\ref{continuation}, the inverse Frobenius operator 
$\varphi^{-1}$ on $\TF_*(X;p)$ gives rise to a meromorphic Frobenius
operator 
$$\xymatrix{
{ \TP_*(X) } \ar[r]^-{\varphi} &
{ \TP_*(X), } \cr
}$$
which is defined and invertible after inverting $p$. We relate this
operator to the geometric Frobenius $\Fr_q^*$ as follows. By
functoriality, the Hodge spectral sequence
$$\textstyle{ E_{i,j}^2 = \bigoplus_{m \in \Z} \lim_{n,F}
H^{-i}(X,W_n\Omega_X^{\hskip1pt j+2m}) \Rightarrow
\TP_{i+j}(X) }$$
is a spectral sequence of graded $\TP_*(k)$-modules. Moreover, on
$\smash{ E_{i,j}^2 }$, and hence, on $\smash{ E_{i,j}^{\infty} }$, the
geometric Frobenius acts as
$$\Fr_q^* = q^w\varphi^r,$$
where $w = j+m$ is the weight and $q = p^r$.

\section*{Acknowledgements}

It is a great pleasure to thank the Hausdorff Research Institute for
Mathematics for its hospitality and support during the Trimester
Program ``Homotopy theory, manifolds, and field theories,'' where part
of the work reported here was conducted. I am also very grateful to Alain
Connes for helpful conversations and for announcing Theorem~\ref{main}
in his paper~\cite{connes1} and to Peter Scholze for explaining the
conjugate spectral sequence to me and for making me fully appreciate
the importance of B\"{o}kstedt's periodicity theorem.

\providecommand{\bysame}{\leavevmode\hbox to3em{\hrulefill}\thinspace}
\providecommand{\MR}{\relax\ifhmode\unskip\space\fi MR }
\providecommand{\MRhref}[2]{%
  \href{http://www.ams.org/mathscinet-getitem?mr=#1}{#2}
}
\providecommand{\href}[2]{#2}


\begin{thebibliography}{10}

\bibitem{SGA4II}
M.~Artin, A.~Grothendieck, and J.~L. Verdier, \emph{Th{\'e}orie des topos et
  cohomologie {\'e}tale des sch{\'e}mas. {T}ome 2}, S{\'e}minaire de
  {G\'e}ometrie {A}lg{\'e}brique du {B}ois-{Marie} 1963--1964 (SGA 4), Lecture
  Notes in Math., vol. 270, Springer-Verlag, New York, 1972.

\bibitem{berthelot}
P.~Berthelot, \emph{Cohomologie cristalline des schemas de caracteristique
  {$p>0$}}, Lecture Notes in Math., vol. 407, Springer-Verlag, New York, 1974.

\bibitem{bhattmorrowscholze1}
B.~Bhatt, M.~Morrow, and P.~Scholze, \emph{Integral {$p$}-adic {H}odge theory
  and topological {H}ochschild homology}, in preparation.

\bibitem{blumbergmandell1}
A.~J. Blumberg and M.~A. Mandell, \emph{The strong {K\"{u}}nneth theorem for
  topological periodic cyclic homology}, arXiv:1706.06846.

\bibitem{blumbergmandell}
\bysame, \emph{Localization theorems in topological {H}ochschild homology and
  topological cyclic homology}, Geom. Topol. \textbf{16} (2012), 1053--1120.

\bibitem{boardman}
J.~M. Boardman, \emph{Conditionally convergent spectral sequences}, Homotopy
  invariant algebraic structures (Baltimore, MD, 1998), Contemp. Math., vol.
  239, Amer. Math. Soc., Providence, RI, 1999, pp.~49--84.

\bibitem{bokstedt}
M.~B\"okstedt, \emph{Topological {H}ochschild homology}, Preprint, Bielefeld
  University, 1985.

\bibitem{bokstedt1}
\bysame, \emph{Topological {H}ochschild homology of {$\mathbb{Z}$} and
  {$\mathbb{Z}/p$}}, Bielefeld, 1985.

\bibitem{borger}
J.~Borger, \emph{The basic geometry of {W}itt vectors,~{I}. {The affine case}},
  Algebra Number Theory \textbf{5} (2011), 231--285.

\bibitem{borger1}
\bysame, \emph{The basic geometry of {W}itt vectors,~{II}. {S}paces}, Math.
  Ann. \textbf{351} (2011), 877--933.

\bibitem{bousfield}
A.~K. Bousfield, \emph{The localization of spectra with respect to homology},
  Topology \textbf{18} (1979), 257--281.

\bibitem{cartan}
H.~Cartan, \emph{S{\'{e}}minaire {1954--1955}. {A}lg\'{e}bre
  d'{E}ilenberg-{M}ac lane et homologie}, Benjamin, New York, 1967.

\bibitem{connes1}
A.~Connes, \emph{An essay on the {R}iemann {H}ypothesis}, Open problems in
  mathematics, Springer, 2016, pp.~225--257.

\bibitem{connesconsani}
A.~Connes and C.~Consani, \emph{Cyclic homology, {S}erre's local factors and
  {$\lambda$}-operations}, J. $K$-theory \textbf{14} (2014), 1--45.

\bibitem{deninger}
C.~Deninger, \emph{Motivic {$L$}-functions and regularized determinants},
  Motives (Seattle, WA, 1991), Proc. Symp. Pure Math., vol.~55, Amer. Math.
  Soc., Providence, RI, 1994.

\bibitem{gabber}
O.~Gabber, \emph{{$K$}-theory of henselian local rings and henselian pairs},
  Algebraic {$K$}-theory, commutative algebra, and algebraic geometry (Santa
  Margherita Ligure, 1989), Contemp. Math., vol. 126, Amer. Math. Soc.,
  Providence, RI, 1992, pp.~59--70.

\bibitem{gh}
T.~Geisser and L.~Hesselholt, \emph{Topological cyclic homology of schemes},
  {$K$}-theory (Seattle, 1997), Proc. Symp. Pure Math., vol.~67, 1999,
  pp.~41--87.

\bibitem{geisserlevine}
T.~Geisser and M.~Levine, \emph{The {$K$}-theory of fields in characteristic
  {$p$}}, Invent. Math. \textbf{139} (2000), 459--493.

\bibitem{gilletthomason}
H.~Gillet and R.~Thomason, \emph{The {$K$}-theory of a strict hensel local ring
  and a theorem of {S}uslin}, J. Pure Appl. Alg. \textbf{34} (1984), 241--254.

\bibitem{greenlees}
J.~P.~C. Greenlees, \emph{Representing {T}ate cohomology of {$G$}-spaces},
  Proc. Edinburgh Math. Soc. (2) \textbf{30} (1987), 435--443.

\bibitem{h}
L.~Hesselholt, \emph{On the $p$-typical curves in {Q}uillen's {$K$}-theory},
  Acta Math. \textbf{177} (1997), 1--53.

\bibitem{h9}
\bysame, \emph{Topological {H}ochschild homology and the de {R}ham-{W}itt
  complex for {$\mathbb{Z}_{(p)}$}-algebras}, Homotopy theory: {R}elations with
  algebraic geometry, group cohomology, and algebraic {$K$}-theory (Evanston,
  IL, 2002), Contemp. Math., vol. 346, Amer. Math. Soc., Providence, RI, 2004,
  pp.~253--259.

\bibitem{h3}
\bysame, \emph{{$K$}-theory of truncated polynomial algebras}, Handbook of
  {$K$}-theory, Springer-Verlag, New York, 2005.

\bibitem{h8}
\bysame, \emph{On the topological cyclic homology of the algebraic closure of a
  local field}, An Alpine Anthology of Homotopy Theory: Proceedings of the
  Second Arolla Conference on Algebraic Topology (Arolla, Switzerland, 2004),
  Contemp. Math., vol. 399, Amer. Math. Soc., Providence, RI, 2006,
  pp.~133--162.

\bibitem{h6}
\bysame, \emph{The big de~{R}ham-{W}itt complex}, Acta Math. \textbf{214}
  (2015), 135--207.

\bibitem{hm}
L.~Hesselholt and I.~Madsen, \emph{On the {$K$}-theory of finite algebras over
  {W}itt vectors of perfect fields}, Topology \textbf{36} (1997), 29--102.

\bibitem{hm4}
\bysame, \emph{On the {$K$}-theory of local fields}, Ann. of Math. \textbf{158}
  (2003), 1--113.

\bibitem{hm3}
\bysame, \emph{On the de~{R}ham-{W}itt complex in mixed characteristic}, Ann.
  Sci. {\'E}cole Norm. Sup. \textbf{37} (2004), 1--43.

\bibitem{illusie}
L.~Illusie, \emph{Complexe de de {R}ham-{W}itt et cohomologie cristalline},
  Ann. Scient. \'Ec. Norm. Sup. (4) \textbf{12} (1979), 501--661.

\bibitem{illusieraynaud}
L.~Illusie and M.~Raynaud, \emph{Les suites spectrales associ{\'{e}}es au
  complexe de de {R}ham-{W}itt}, Inst. Hautes \'{E}tudes Sci. Publ. Math.
  \textbf{57} (1983), 73--212.

\bibitem{jardine}
J.~F. Jardine, \emph{Local homotopy theory}, Springer Monographs in
  Mathematics, Springer, New York, 2015.

\bibitem{nikolausscholze}
T.~Nikolaus and P.~Scholze, \emph{On topological cyclic homology},
  arXiv:1707.01799.

\bibitem{suslin}
A.~A. Suslin, \emph{On the {$K$}-theory of algebraically closed fields},
  Invent. Math. \textbf{73} (1983), 241--245.

\bibitem{suslin1}
\bysame, \emph{On the {$K$}-theory of local fields}, J. Pure Appl. Alg.
  \textbf{34} (1984), 304--318.

\bibitem{tabuada}
G.~Tabuada, \emph{Finite generation of the numerical {G}rothendieck group},
  arXiv:1704.06252.

\bibitem{thomason}
R.~W. Thomason, \emph{Algebraic {$K$}-theory and \'etale cohomology}, Ann.
  Scient. \'Ecole Norm. Sup. \textbf{13} (1985), 437--552.

\bibitem{tsalidis}
S.~Tsalidis, \emph{Topological {H}ochschild homology and the homotopy descent
  problem}, Topology \textbf{37} (1998), 913--934.

\end{thebibliography}
\end{document}